\documentclass[a4paper]{article}
\usepackage[utf8]{inputenc}
\usepackage{hhline}
\usepackage{lmodern}
\usepackage{array}
\usepackage{graphicx}
\usepackage{stmaryrd}
\usepackage{subfig}
\usepackage{amssymb, amsfonts, amsthm, amsmath}
\usepackage{enumerate}
\usepackage{enumitem}
\usepackage{bbm}
\usepackage{enumitem}
\usepackage{svg}
\usepackage{import}
\usepackage{mathtools}
\mathtoolsset{showonlyrefs}
\usepackage{hyperref}

\usepackage[english]{babel}

\usepackage{tikz}
\usetikzlibrary{decorations, calc}
\usetikzlibrary{arrows.meta}
\usetikzlibrary{decorations.pathreplacing}

\tikzset{individu/.style={draw,thick}}

\numberwithin{equation}{section}

\theoremstyle{plain}
\newtheorem{theorem}{Theorem}[section]
\newtheorem{corollary}[theorem]{Corollary}
\newtheorem{conjecture}[theorem]{Conjecture}
\newtheorem{lemma}[theorem]{Lemma}
\newtheorem{proposition}[theorem]{Proposition}

\theoremstyle{definition}
\newtheorem{definition}[theorem]{Definition}

\newtheorem{question}[theorem]{Question}

\theoremstyle{remark}
\newtheorem{remark}[theorem]{Remark}

\newcommand{\R}{\mathbb{R}}
\newcommand{\C}{\mathbb{C}}

\newcommand{\calD}{\mathcal{D}}
\newcommand{\calC}{\mathcal{C}}

\newcommand{\calB}{\mathcal{B}}

\newcommand{\Kel}{\mathrm{K}}
\newcommand{\Fel}{\mathrm{F}}

\renewcommand{\Re}{\operatorname{Re}}
\renewcommand{\Im}{\operatorname{Im}}

\DeclareMathOperator{\cotan}{\mathrm{cotan}}

\renewcommand{\tilde}[1]{\widetilde{#1}}
\renewcommand{\epsilon}{\varepsilon}
\renewcommand{\phi}{\varphi}

\DeclareMathOperator{\sgn}{\mathrm{sgn}}

\DeclareMathOperator{\cn}{\mathrm{cn}}
\DeclareMathOperator{\sn}{\mathrm{sn}}
\DeclareMathOperator{\dn}{\mathrm{dn}}
\let\sc\relax
\DeclareMathOperator{\sc}{\mathrm{sc}}

\DeclareMathOperator{\nc}{\mathrm{nc}}
\DeclareMathOperator{\ns}{\mathrm{ns}}
\DeclareMathOperator{\cs}{\mathrm{cs}}
\DeclareMathOperator{\ds}{\mathrm{ds}}
\DeclareMathOperator{\cd}{\mathrm{cd}}
\DeclareMathOperator{\sd}{\mathrm{sd}}
\DeclareMathOperator{\nd}{\mathrm{nd}}
\DeclareMathOperator{\am}{\mathrm{am}}

\title{Cube moves for $s$-embeddings and $\alpha$-realizations}
\author{Paul Melotti\footnote{Faculté des Sciences d'Orsay, Université Paris-Saclay,
    91405 Orsay Cedex, France.\hfill\break
  \textit{E-mail address}: \texttt{paul.melotti at
    universite-paris-saclay.fr}}
\and Sanjay Ramassamy\footnote{Université Paris-Saclay, CNRS, CEA, Institut de physique théorique, 91191 Gif-sur-Yvette, France.\hfill\break
  \textit{E-mail address}: \texttt{sanjay.ramassamy at ipht.fr}}
\and Paul Th\'evenin\footnote{Uppsala Universitet, Lägerhyddsvägen 1, 752 37 Uppsala, Sweden.\hfill\break
  \textit{E-mail address}: \texttt{paul.thevenin at math.uu.se}}
}
\date{\today}

\newcommand{\dbc}{\Upsilon^{\times}}

\begin{document}

\maketitle

\begin{abstract}
Chelkak introduced $s$-embeddings as tilings by tangential quads which provide the right setting to study the Ising model with arbitrary coupling constants on arbitrary planar graphs. We prove the existence and uniqueness of a local transformation for $s$-embeddings called the cube move, which consists in flipping three quadrilaterals in such a way that the resulting tiling is also in the class of $s$-embeddings. In passing, we give a new and simpler formula for the change in coupling constants for the Ising star-triangle transformation which is conjugated to the cube move for $s$-embeddings. We introduce more generally the class of $\alpha$-embeddings as tilings of a portion of the plane by quadrilaterals such that the side lengths of each quadrilateral $ABCD$ satisfy the relation $AB^\alpha+CD^\alpha=AD^\alpha+BC^\alpha$, providing a common generalization for harmonic embeddings adapted to the study of resistor networks ($\alpha=2$) and for $s$-embeddings ($\alpha=1$). We investigate existence and uniqueness properties of the cube move for these $\alpha$-embeddings.
\end{abstract}

\section{Introduction}

The star-triangle transformation was first introduced by Kennelly in the context of resistor networks \cite{Kennelly}, as a local transformation which does not change the electrical properties of the network (such as the equivalent resistance) outside of the location where the transformation is performed. It consists in replacing a vertex of degree $3$ by a triangle as on Figure~\ref{fig:st_ising}, and the conductances after the transformation are given as some rational functions of the conductances before the transformation. It follows from the classical connection between resistor networks and random walks \cite{Kakutani} that applying this star-triangle transformation to a graph with weights on the edges also preserves the probabilistic properties of the random walk on this graph. Conversely, a triangle can be made into a star, and the terminology ``star-triangle transformation'' is often used to refer to both of these
operations.

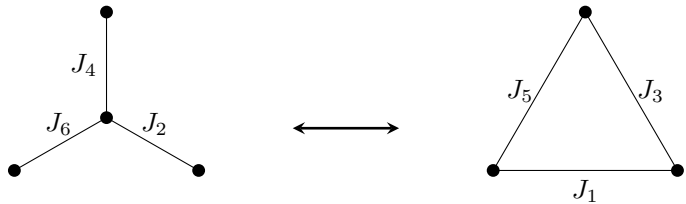
\begin{figure}[h]
  \centering
  \begin{tikzpicture}[scale=0.7]

\begin{scope}[xshift=9cm]
  \coordinate (hh) at (0 ,2) ;
  \coordinate (bd) at (1.732,-1) ;
  \coordinate (bg) at (-1.732,-1) ;
  \coordinate (hd) at ($(hh)+(bd)$) ;
  \coordinate (hg) at ($(hh)+(bg)$) ;
  \coordinate (bb) at ($(bd)+(bg)$) ;
  \node [draw=black, fill=black,thick,circle,inner sep=0pt,minimum size=4pt] at (hh) {};
  \node [draw=black, fill=black,thick,circle,inner sep=0pt,minimum size=4pt] at (bd) {};
  \node [draw=black, fill=black,thick,circle,inner sep=0pt,minimum size=4pt] at (bg) {};
  \draw (hh) -- (bg) -- (bd) -- (hh);
  \draw (-0.8,0.5) node [left] {$J_5$};
  \draw (0.8,0.5) node [right] {$J_3$};
  \draw (0,-1) node [below] {$J_1$};
\end{scope}

\begin{scope}[yshift=0cm]
  \begin{scope}[xshift=0cm]
  \coordinate (hh) at (0 ,2) ;
  \coordinate (bd) at (1.732,-1) ;
  \coordinate (bg) at (-1.732,-1) ;
  \coordinate (hd) at ($(hh)+(bd)$) ;
  \coordinate (hg) at ($(hh)+(bg)$) ;
  \coordinate (bb) at ($(bd)+(bg)$) ;
  \coordinate (cc) at ($(hd) + (bb)- (bd)$);
  \node [draw=black, fill=black,thick,circle,inner sep=0pt,minimum size=4pt] at (hh) {};
  \node [draw=black, fill=black,thick,circle,inner sep=0pt,minimum size=4pt] at (bd) {};
  \node [draw=black, fill=black,thick,circle,inner sep=0pt,minimum size=4pt] at (bg) {};
  \node [draw=black, fill=black,thick,circle,inner sep=0pt,minimum size=4pt] at (cc) {};
  \draw (hh) -- (cc);
  \draw (bg) -- (cc);
  \draw (bd) -- (cc);
  \draw (0.9,-0.5) node [above] {$J_2$};
  \draw (-0.9,-0.5) node [above] {$J_6$};
  \draw (0,1) node [left] {$J_4$};
\end{scope}
\begin{scope}[xshift=3.5cm, yshift=-0.2cm]
  \draw[>=stealth,<->, line width = 1pt] (0,0) -- (2,0);
\end{scope}
\end{scope}

\end{tikzpicture}
  \caption{Star-triangle transformation.}
  \label{fig:st_ising}
\end{figure}

Another probabilistic model known to possess a star-triangle transformation is the Ising model, a celebrated model of magnetization which samples a random configuration of spins living at the vertices
of a graph; this property appears in \cite{Onsager,Wannier}, see also Chapter~6 of \cite{Baxter:exactly}. The probability distribution of this configuration depends on coupling constants attached to the edges of the graph. When these coupling constants are all positive (which is called the ferromagnetic regime), one can perform a local transformation of the graph as on Figure~\ref{fig:st_ising} without changing the correlations of spins outside of the location where the transformation is performed \cite{Onsager,Wannier}. Note however that the formulas relating the coupling constants before and after the transformation are not the same as those for the conductances in resistor networks.

Let $G$ be a planar graph, that is, a graph that can be embedded in the plane - or equivalently in the sphere. We denote by $G^\diamond$ its diamond graph, whose vertex set is composed of the
vertices and dual vertices of $G$ and whose faces are all
quadrilaterals, one for each edge of $G$ (see
Figure~\ref{fig:diamond}). We note that planar graphs are graphs that can be embedded in the plane but which do not come with a distinguished embedding. From now on we will assume that every planar graph $G$ is $3$-connected, which implies \cite{Whitney} that it possesses a unique embedding up to homeomorphisms of the sphere. In particular the faces of $G$ are well-defined. Hence its diamond graph $G^\diamond$ is well-defined, is also $3$-connected and has well-defined faces.

\begin{figure}[h]
  \centering
  \def\svgwidth{6cm}
  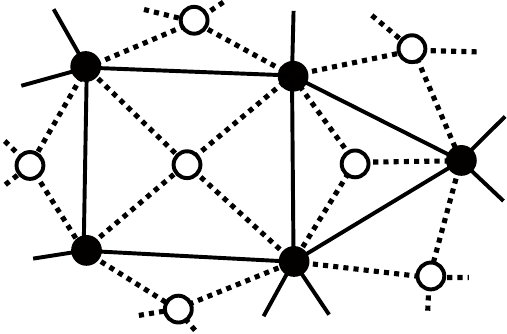
  \caption{A piece of a planar graph $G$ (black dots and solid
    lines), its dual vertices (white dots) and its diamond graph
    $G^{\diamond}$ (dotted lines).}
  \label{fig:diamond}
\end{figure}

To each of the two models described above (random walk and Ising model) is associated a class of graph embeddings. More precisely, if $G$ denotes a planar graph carrying positive edge weights (conductances for the random walk, coupling constants for the Ising model), one embeds in the plane its diamond graph $G^\diamond$.

The embeddings associated to random walks are called the Tutte embeddings or harmonic embeddings \cite{Tutte:emb} and have the property that the
faces of $G^\diamond$ are embedded as orthodiagonal quads, that is, their diagonals are perpendicular. The conductance of an edge of $G$ corresponding to an orthodiagonal quad embedding of a face of $G^\diamond$ is given by a ratio of the lengths of the diagonals of the quad. The embeddings associated to the ferromagnetic Ising model were introduced by Chelkak \cite{Chelkak,Chelkak2} under the name of $s$-embeddings; in this case the faces of $G^\diamond$ are embedded as tangential quads, which are quads admitting a circle tangential to the four sides. The coupling constant of an edge $e$ of $G$ corresponding to a tangential quad embedding of a face of $G^\diamond$ is given by
\[
  J_e = \frac12 \ln\left(\frac{1+\sin\theta_e}{\cos\theta_e}\right),
\]
where
\[
    \tan^2\theta_e = \frac{\cotan \delta + \cotan \beta}{\cotan \alpha
      + \cotan \gamma}.
\]
and $\alpha,\beta,\gamma,\delta$ denote the half-angles of the tangential quad, as shown on Figure~\ref{fig:1quad}. In the special case when the tangential quad is a rhombus, the angle $\theta_e$ arises as the half-angle of a corner of the quad corresponding to a primal vertex. In the general case $\theta_e$ does not seem to have a geometric realization as an angle.

\begin{figure}[h]
  \centering
  \def\svgwidth{10cm}
\begingroup%
  \makeatletter%
  \providecommand\color[2][]{%
    \errmessage{(Inkscape) Color is used for the text in Inkscape, but the package 'color.sty' is not loaded}%
    \renewcommand\color[2][]{}%
  }%
  \providecommand\transparent[1]{%
    \errmessage{(Inkscape) Transparency is used (non-zero) for the text in Inkscape, but the package 'transparent.sty' is not loaded}%
    \renewcommand\transparent[1]{}%
  }%
  \providecommand\rotatebox[2]{#2}%
  \newcommand*\fsize{\dimexpr\f@size pt\relax}%
  \newcommand*\lineheight[1]{\fontsize{\fsize}{#1\fsize}\selectfont}%
  \ifx\svgwidth\undefined%
    \setlength{\unitlength}{358.49132628bp}%
    \ifx\svgscale\undefined%
      \relax%
    \else%
      \setlength{\unitlength}{\unitlength * \real{\svgscale}}%
    \fi%
  \else%
    \setlength{\unitlength}{\svgwidth}%
  \fi%
  \global\let\svgwidth\undefined%
  \global\let\svgscale\undefined%
  \makeatother%
  \begin{picture}(1,0.26451463)%
    \lineheight{1}%
    \setlength\tabcolsep{0pt}%
    \put(0,0){\includegraphics[width=\unitlength,page=1]{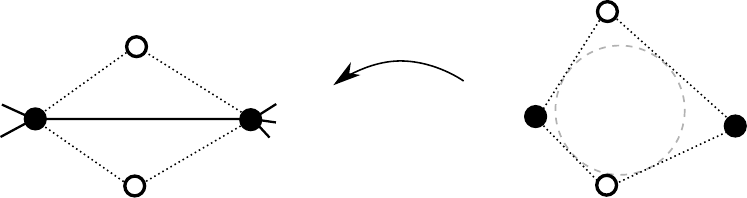}}%
    \put(0.14985405,0.11322556){\color[rgb]{0,0,0}\makebox(0,0)[lt]{\lineheight{1.25}\smash{\begin{tabular}[t]{l}$e$\end{tabular}}}}%
    \put(0,0){\includegraphics[width=\unitlength,page=2]{1emb.pdf}}%
    \put(0.79180727,0.17785592){\color[rgb]{0,0,0}\makebox(0,0)[lt]{\lineheight{1.25}\smash{\begin{tabular}[t]{l}$2\delta$\end{tabular}}}}%
    \put(0,0){\includegraphics[width=\unitlength,page=3]{1emb.pdf}}%
    \put(0.75553928,0.10291363){\color[rgb]{0,0,0}\makebox(0,0)[lt]{\lineheight{1.25}\smash{\begin{tabular}[t]{l}$2\alpha$\end{tabular}}}}%
    \put(0.89892303,0.09880417){\color[rgb]{0,0,0}\makebox(0,0)[lt]{\lineheight{1.25}\smash{\begin{tabular}[t]{l}$2\gamma$\end{tabular}}}}%
    \put(0.80031057,0.06667545){\color[rgb]{0,0,0}\makebox(0,0)[lt]{\lineheight{1.25}\smash{\begin{tabular}[t]{l}$2\beta$\end{tabular}}}}%
  \end{picture}%
\endgroup%

  \caption{Relation between an abstract weighted graph and an $s$-embedding.}
  \label{fig:1quad}
\end{figure}

Existence and uniqueness questions for both harmonic embeddings and $s$-embeddings may be asked either for finite weighted graphs or infinite weighted graphs periodic in two directions. Complete answers are known for harmonic embeddings in both cases and for $s$-embeddings in the periodic case, see \cite{KLRR,Chelkak2} for a discussion. However, such questions are not relevant for our purposes, as we will start from one given embedding and discuss whether or not there exists another embedding related by a local transformation.

To avoid unnecessary complications related to boundary issues in the case of finite graphs, we will always assume that the edges involved in our star-triangle transformations are not boundary edges. Note that one could consider a broader framework including such boundary edges, provided that one replaces the notion of a dual graph by that of the augmented dual graph, which has several dual vertices associated with the outer face, one between each pair of consecutive boundary vertices, see e.g. \cite{KLRR}.

A harmonic embedding or an $s$-embedding of the diamond graph $G^\diamond$ determines a drawing of $G$, since the vertices of $G$ form a subset of the vertices of $G^\diamond$. However a drawing of $G$ does not uniquely determine a harmonic embedding or an $s$-embedding.

Such embeddings provide the appropriate geometric setting to observe conformally invariant objects such as Brownian motion or SLE processes when taking the scaling limits of random walks and Ising models on generic planar graphs \cite{BergerBiskup,Chelkak,Chelkak2}. They
generalize isoradial embeddings (embeddings of the faces of $G^\diamond$ as rhombi), for which specific techniques can
lead to proofs of scaling limits (for instance in \cite{Kenyon:intro,
  ChelkakSmirnov1, ChelkakSmirnov2, GrimmettManolescu,
  BoutillierDeTiliereRaschel} and others),
but which correspond to specific choices of weights on an already
embedded graph.

\begin{figure}[h]
  \centering
  \def\svgwidth{6cm}
  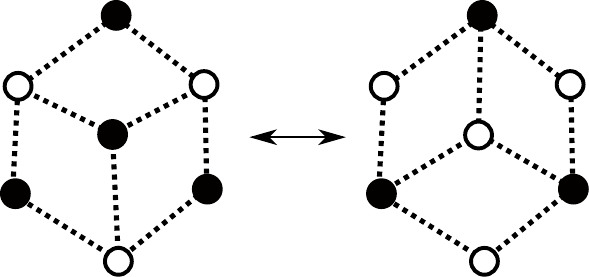
  \caption{The transformation on $G^{\diamond}$ induced by the
    star-triangle transformation on $G$: a \emph{cube move}.}
  \label{fig:cube_diamond}
\end{figure}

The star-triangle move for random walks (resp. the Ising model)
translates into a geometric local move for harmonic embeddings
(resp. for $s$-embeddings), whereby three orthodiagonal quads
(resp. tangential quads) sharing a vertex as on the left-hand side
picture of Figure~\ref{fig:cube_diamond} get erased
and replaced by three other orthodiagonal quads (resp. tangential
quads) with the same hexagonal outer boundary, as on the right-hand side picture of Figure~\ref{fig:cube_diamond}. We call this geometric local move a \emph{cube move}, phrase which was
originally introduced in \cite{KenyonPemantle_minors} only with a combinatorial meaning (and not a geometric one). We emphasize that cube moves for us are geometric local moves which apply to tilings of the plane by quads, while star-triangle transformations are combinatorial local moves which apply to graphs with edge weights.

The existence and uniqueness of the cube move for harmonic embeddings
follows from a classical theorem of planar geometry called Steiner's
theorem \cite{KS,KLRR}. In this article, we show the existence and uniqueness of the cube move for $s$-embeddings (see definitions in Section \ref{sec:def} and Theorem \ref{theo:1quad} for the exact statement).

\begin{figure}[h]
  \centering
  \def\svgwidth{11cm}
  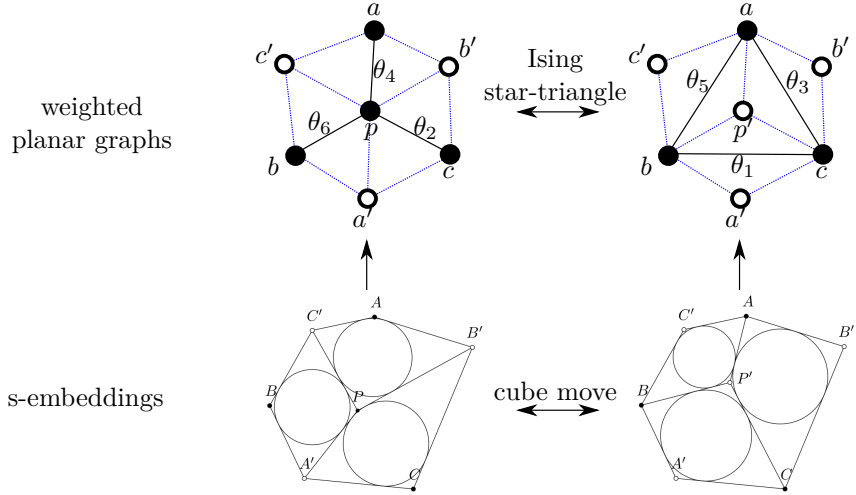
  \caption{The star-triangle transformation for the Ising model is
    conjugated to a local geometric transformation called a \emph{cube
      move}, via Chelkak's $s$-embeddings introduced in
    \cite{Chelkak,Chelkak2}. Notice that in the embedded graphs, only the
    central points $P$ and $P'$ differ.}
  \label{fig:diagram}
\end{figure}

\begin{theorem}
For any $s$-embedding with the combinatorics of the graph on the left-hand side of Figure~\ref{fig:cube_diamond}, erase the central black point and keep all the other points fixed. Then there exists a unique central white point that gives an $s$-embedding with the combinatorics of the graph on the right-hand side of Figure~\ref{fig:cube_diamond}.
\end{theorem}

Moreover it follows from the construction in Section~\ref{sec:Ising} of the $s$-embedding after the cube move that the coupling constants associated with this new $s$-embedding are the same as those obtained by applying the star-triangle transformation to the coupling constants associated with the $s$-embedding before the cube move.
In other words, the star-triangle transformation of the Ising model
only has a local effect on $s$-embeddings. This is illustrated by a
commutative diagram in Figure~\ref{fig:diagram}. Our proof of
existence relies only on \cite[Proposition $4.7$]{Chelkak} by Chelkak, which connects a linear system on the graph (called \emph{propagation
  equations}) to the construction of $s$-embeddings. We prove
uniqueness via self-contained geometric arguments (see
Section~\ref{sec:Ising}).

In passing we prove new and simpler formulas for the transformation of coupling constants under the Ising star-triangle transformation:

\begin{theorem}
\label{thm:newformula}
Denote by $J_2,J_4,J_6$ (resp. $J_1,J_3,J_5$) the coupling
constants before (resp. after) the star-triangle transformation for
the Ising model, as on Figure~\ref{fig:st_ising}. For every $1\leq i \leq 6$ let $\theta_i$ be the unique angle in $(0,\tfrac{\pi}{2})$ such that $J_i = \frac12 \ln\left(\frac{1+\sin\theta_i}{\cos\theta_i}\right)$. Then
\begin{align}
\forall i \in \{2,4,6\}, & \ \cos(\theta_i) = \frac{\sin\theta_{i+3} \cos \theta_{i+1} \cos
      \theta_{i+5}}{\sin \theta_{i+3} + \sin \theta_{i+1} \sin
      \theta_{i+5}},   \label{eq:st_theta2} \\
\forall i \in \{1,3,5\}, & \ \sin(\theta_i) =
     \frac{\cos\theta_{i+3} \sin \theta_{i+1} \sin
      \theta_{i+5}}{\cos \theta_{i+3} + \cos \theta_{i+1} \cos
      \theta_{i+5}},   \label{eq:st_theta1}
\end{align}
where the indices are considered modulo $6$.
\end{theorem}

Formulas \eqref{eq:st_theta2} and \eqref{eq:st_theta1} are considerably less involved than the
classical transformation procedure which is described in Proposition
\ref{prop:Baxter}. We stress that the relation between the coupling constants before and after the star-triangle transformation is the same as in \cite{Onsager,Wannier,Baxter:exactly}; what is new is the simpler expression for that relation.

\medskip

In addition to the results on local transformations for $s$-embeddings and the Ising model presented above, the second main contribution of this article involves a new class of embeddings, called $\alpha$-embeddings, which we introduce a one-parameter common generalization of harmonic embeddings and $s$-embeddings. Given $\alpha\in\R^*$, we call a quadrilateral $ABCD$ drawn on the plane an \emph{$\alpha$-quad} if its side lengths satisfy $AB^\alpha+CD^\alpha=AD^\alpha+BC^\alpha$. The quadrilateral $ABCD$ may be non-convex or even have its edges intersecting away from their endpoints. An \emph{$\alpha$-embedding} of a graph $G$ is an embedding of its
diamond graph $G^\diamond$ such that each face of $G^\diamond$ is
drawn as an $\alpha$-quad.
This definition is firstly motivated by the introduction of a unified
framework for properties of harmonic and $s$-embeddings, which
correspond respectively to the cases $\alpha=2$ and
$\alpha=1$. Secondly, the celebrated isoradial embeddings are $\alpha$-embeddings for every $\alpha$, since rhombi are $\alpha$-quads for every $\alpha$. More generally, kite embeddings are $\alpha$-embeddings for every $\alpha$, see Section~\ref{sec:space}. Thirdly, as we shall see, there is a whole range of
parameters for which these quadrilaterals satisfy a version of a cube
move. This property suggests the presence of an integrable system, and it would be remarkable to find a family of such systems indexed by a continuous parameter.

To state this version of the cube move, we introduce the weaker notion
of an $\alpha$-realization of $G$ as a map from the vertices of
$G^\diamond$ to the plane such that each face of $G^\diamond$ is
mapped to an $\alpha$-quad, with edges possibly intersecting. We also
provide a definition of the above notions in the cases where
$\alpha\in\{-\infty,0,+\infty\}$. In addition, we show that, for
$\alpha>1$, the cube move is possible for $\alpha$-realizations (although the solution may not be unique). See Theorem \ref{theo:>1} for a
precise statement.

\begin{theorem}
  \label{theo:>1intro}
  Let $\alpha>1$ be a real number.
  For any $\alpha$-realization with the combinatorics of the graph on the
  left-hand side of Figure~\ref{fig:cube_diamond}, erase the central
  black point and keep all the other points fixed. Then there exists a
  central white point that gives an $\alpha$-realization with the
  combinatorics of the graph on the right-hand side of Figure~\ref{fig:cube_diamond}.
\end{theorem}

One may wonder whether a probabilistic model can be associated to $\alpha$-embeddings beyond the cases where $\alpha$ is $1$ or $2$. In other words, is there a way to define the interaction constants of a probabilistic model from the local geometry of $\alpha$-embeddings in such a way that the cube move for $\alpha$-embeddings is conjugated to the star-triangle move for the probabilistic model? The FK-percolation model is a common generalization of the Ising model and of spanning trees (which have the same star-triangle transformation as random walks). This model also has a star-triangle transformation; however, it can be applied only when the weights satisfy some local constraint  \cite{Kenyon:intro,DCLM}, whereas the star-triangle transformation for the Ising model and spanning trees holds without any condition on the weights. We did not succeed in relating $\alpha$-embeddings to FK-percolation for arbitrary weights that would satisfy the local constraint and do not expect such a general connection to hold true. Nevertheless there exists a subclass of weights among those satisfying the local constraint which can be associated to isoradial embeddings, that is, embeddings where the quadrilaterals are rhombi. In that case, the star-triangle transformation corresponds to a cube move for rhombus tilings, which exists and is unique.

\subsection*{Organization of the paper}

We define in Section~\ref{sec:def} our main object of interest,
$\alpha$-quads; we introduce $\alpha$-embeddings and
$\alpha$-realizations, before defining their cube
move. Section~\ref{sec:Ising} is devoted to the specific case
$\alpha=1$, where we recall the connection with the Ising model and
prove the existence and uniqueness of the cube move for properly
embedded graphs. We also prove Theorem~\ref{thm:newformula}.
In Section~\ref{sec:>1}, we show that for $\alpha>1$, such a move is always possible under weaker conditions, which leads to the proof of Theorem~\ref{theo:>1intro}.
Finally in Section~\ref{sec:space} we investigate some more
basic geometric properties of $\alpha$-quads and we propose a general
framework to study a broader class of quadrilaterals.  Appendix~\ref{sec:elliptic} contains a brief introduction to Jacobi elliptic functions, which are one of our main tools in the proofs of Section~\ref{sec:Ising}.

\section{Definitions}
\label{sec:def}

This first section is devoted to the definition of the so-called
$\alpha$-quads --- which are the main object of interest of this paper --- of their embeddings and of their realizations. We introduce an operation on
them called a \emph{cube move}, before defining an important tool to
prove this property, which we call
\emph{construction curves}.

In what follows, we
denote by $\R^*$ the set $\R \setminus \{0\}$ and by $\overline \R$ the set
$\R \cup \{-\infty,+\infty\}$ of extended real numbers. For two points
$A,B$ in the plane, $AB$ denotes the Euclidean distance between $A$ and $B$,
and we use
a dot to write products of lengths such as $AB.CD$.

\subsection{$\alpha$-embeddings and realizations}

Let us start by defining the notion of $\alpha$-quads.

\begin{definition}
\label{def:alpha}
A quadrilateral $ABCD$ is called an \emph{$\alpha$-quad} for $\alpha \in \R^*$ if
\begin{equation}
\label{eq:quadlengths}
AB^\alpha + CD^\alpha = AD^\alpha + BC^\alpha.
\end{equation}
It is called a $0$-quad if
\begin{equation}
\label{eq:alpha0}
AB.CD=AD.BC,
\end{equation}
a $+\infty$-quad if
\begin{equation}
\max(AB,CD) = \max(AD,BC)
\end{equation}
and a $-\infty$-quad if
\begin{equation}
\min(AB,CD) = \min(AD,BC).
\end{equation}
\end{definition}

Note that the $\alpha$-quads are not required to be convex nor even proper (see Definition~\ref{def:proper}), meaning that edges may intersect away from their endpoints.

One can immediately notice that $\alpha$-quads (for
$\alpha\in\overline\R$) are simply quadrilaterals $ABCD$ such that
$f_\alpha(AB,CD) = f_\alpha(AD,BC)$, where for $x,y>0$ we set
$f_\alpha(x,y)=x^\alpha+y^\alpha$ if $\alpha \in \R^*$, $f_0(x,y)=xy$,
$f_{+\infty}(x,y)=\max(x,y)$ and $f_{-\infty}(x,y)=\min(x,y)$. Such a
notation calls for a generalization from $f_\alpha$ to any homogeneous
symmetric function $f$ of two variables, which is done at the end of Section~\ref{sec:space}.

Of particular interest are the following three families of $\alpha$-quads, which were already defined in a different context:
\begin{itemize}
\item $1$-quads correspond to \emph{tangential quads}, i.e. quads such that there is a circle tangential to their four sides (see e.g. \cite{Josefsson2}) ;

\item $2$-quads correspond to \emph{orthodiagonal quads}, i.e. quads whose diagonals are perpendicular (see e.g. \cite{Josefsson}) ;

\item $0$-quads are known under the name of \emph{balanced quads} \cite{Josefsson2} and contain a well-studied class of quads, the \emph{harmonic quads}, which are defined as cyclic $0$-quads (that is, $0$-quads inscribed in a circle, see \cite{Atzema} for details).
\end{itemize}

\begin{definition}
\label{def:proper}
  For $n\geq 3$, an $n$-tuple of distinct points $A_1,\dots,A_n$ is said to be a
  \emph{proper} polygon if the line
  segments $[A_1A_2], [A_2A_3], \dots, [A_{n-1}A_n], [A_nA_1]$ do not
  intersect except possibly at their endpoints. In other words, the
  closed piecewise linear curve
  $A_1,A_2,\dots,A_n,A_1$ is a Jordan curve.

  A proper polygon is said to be \emph{positively} (resp. \emph{negatively})
    \emph{oriented} if the points on its boundary oriented counterclockwise
  (resp. clockwise) are $A_1,A_2,\dots,A_n$ in this order.
\end{definition}

Let $G$ be a planar graph, finite or infinite. Denote its dual graph by $G^*$, whose vertices are faces of $G$ and where we connect two dual vertices if the two corresponding faces share an edge. We construct the graph $G^{\diamond}$ as the bipartite graph whose black (resp. white) vertices are the vertices of $G$ (resp. $G^*$) and where an edge connects a black vertex to a white vertex whenever the black vertex is on the boundary of the face associated with the white vertex. In particular, all the faces of $G^{\diamond}$ are quadrilaterals, see Figure~\ref{fig:diamond}.

\begin{definition}
Let $\alpha \in \overline{\R}$. An \emph{$\alpha$-embedding} of $G$ is
an embedding of $G^{\diamond}$ in the plane such that every internal face of
$G^{\diamond}$ is an $\alpha$-quad.
\end{definition}

The fact that $G^{\diamond}$ is embedded in the plane implies that the $\alpha$-quads corresponding to internal faces are proper. In particular, for embedded graphs, edges cannot collapse to a point, two distinct edges cannot meet outside of their endpoints, faces have non-empty interiors and two distinct faces have disjoint interiors. This setting is the one preferred for the study of statistical mechanical models such as spanning trees and the Ising model. However, in the more general setting we consider, we need to allow drawings of graphs that are not necessarily embeddings.

\begin{definition}
Let $\alpha \in \overline{\R}$. An \emph{$\alpha$-realization} of $G$ is defined to be any map from the vertices of $G^{\diamond}$ to the plane such that every face of $G^{\diamond}$ is mapped to an $\alpha$-quad.
\end{definition}

See Figure \ref{fig:embeddingrealization} for an example of a $4$-embedding and a $4$-realization of the same graph.

\begin{figure}[htbp]
\begin{center}
\includegraphics[width=4in]{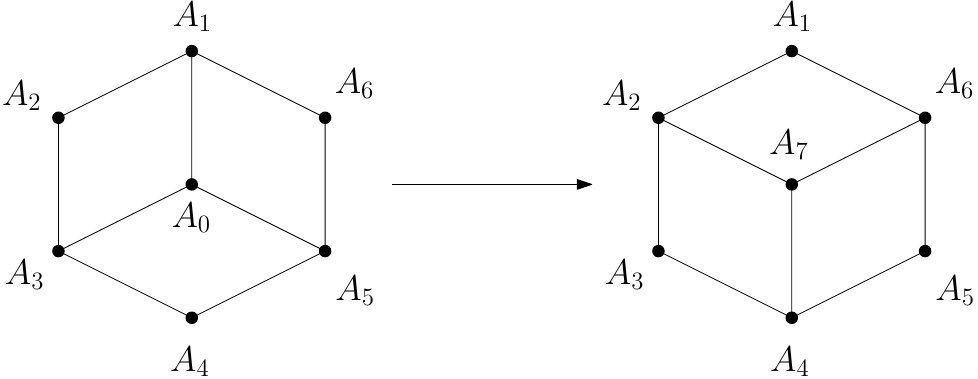}
\end{center}
\caption{The combinatorics of a cube move for $\alpha$-embeddings and $\alpha$-realizations.}
\label{fig:YDelta}
\end{figure}

\begin{figure}[h]
  \centering
  \def\svgwidth{5.1cm}
\begingroup%
  \makeatletter%
  \providecommand\color[2][]{%
    \errmessage{(Inkscape) Color is used for the text in Inkscape, but the package 'color.sty' is not loaded}%
    \renewcommand\color[2][]{}%
  }%
  \providecommand\transparent[1]{%
    \errmessage{(Inkscape) Transparency is used (non-zero) for the text in Inkscape, but the package 'transparent.sty' is not loaded}%
    \renewcommand\transparent[1]{}%
  }%
  \providecommand\rotatebox[2]{#2}%
  \newcommand*\fsize{\dimexpr\f@size pt\relax}%
  \newcommand*\lineheight[1]{\fontsize{\fsize}{#1\fsize}\selectfont}%
  \ifx\svgwidth\undefined%
    \setlength{\unitlength}{222.12371225bp}%
    \ifx\svgscale\undefined%
      \relax%
    \else%
      \setlength{\unitlength}{\unitlength * \real{\svgscale}}%
    \fi%
  \else%
    \setlength{\unitlength}{\svgwidth}%
  \fi%
  \global\let\svgwidth\undefined%
  \global\let\svgscale\undefined%
  \makeatother%
  \begin{picture}(1,0.82331768)%
    \lineheight{1}%
    \setlength\tabcolsep{0pt}%
    \put(0.33252824,0.79276787){\color[rgb]{0,0,0}\makebox(0,0)[t]{\lineheight{1.25}\smash{\begin{tabular}[t]{c}$A_1$\end{tabular}}}}%
    \put(0.05882095,0.34214298){\color[rgb]{0,0,0}\makebox(0,0)[t]{\lineheight{1.25}\smash{\begin{tabular}[t]{c}$A_2$\end{tabular}}}}%
    \put(0.15652544,0.10005281){\color[rgb]{0,0,0}\makebox(0,0)[t]{\lineheight{1.25}\smash{\begin{tabular}[t]{c}$A_3$\end{tabular}}}}%
    \put(0.80897442,0.00946889){\color[rgb]{0,0,0}\makebox(0,0)[t]{\lineheight{1.25}\smash{\begin{tabular}[t]{c}$A_4$\end{tabular}}}}%
    \put(0.74166683,0.50148766){\color[rgb]{0,0,0}\makebox(0,0)[t]{\lineheight{1.25}\smash{\begin{tabular}[t]{c}$A_5$\end{tabular}}}}%
    \put(0.94117901,0.5212679){\color[rgb]{0,0,0}\makebox(0,0)[t]{\lineheight{1.25}\smash{\begin{tabular}[t]{c}$A_6$\end{tabular}}}}%
    \put(0.6238147,0.311746){\color[rgb]{0,0,0}\makebox(0,0)[t]{\lineheight{1.25}\smash{\begin{tabular}[t]{c}$A_0$\end{tabular}}}}%
    \put(0,0){\includegraphics[width=\unitlength,page=1]{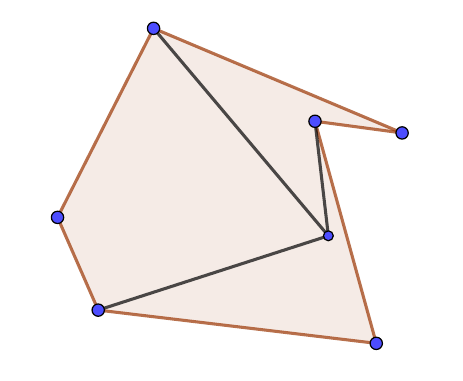}}%
  \end{picture}%
\endgroup%

  \def\svgwidth{5.1cm}
\begingroup%
  \makeatletter%
  \providecommand\color[2][]{%
    \errmessage{(Inkscape) Color is used for the text in Inkscape, but the package 'color.sty' is not loaded}%
    \renewcommand\color[2][]{}%
  }%
  \providecommand\transparent[1]{%
    \errmessage{(Inkscape) Transparency is used (non-zero) for the text in Inkscape, but the package 'transparent.sty' is not loaded}%
    \renewcommand\transparent[1]{}%
  }%
  \providecommand\rotatebox[2]{#2}%
  \newcommand*\fsize{\dimexpr\f@size pt\relax}%
  \newcommand*\lineheight[1]{\fontsize{\fsize}{#1\fsize}\selectfont}%
  \ifx\svgwidth\undefined%
    \setlength{\unitlength}{222.12371225bp}%
    \ifx\svgscale\undefined%
      \relax%
    \else%
      \setlength{\unitlength}{\unitlength * \real{\svgscale}}%
    \fi%
  \else%
    \setlength{\unitlength}{\svgwidth}%
  \fi%
  \global\let\svgwidth\undefined%
  \global\let\svgscale\undefined%
  \makeatother%
  \begin{picture}(1,0.82331768)%
    \lineheight{1}%
    \setlength\tabcolsep{0pt}%
    \put(0,0){\includegraphics[width=\unitlength,page=1]{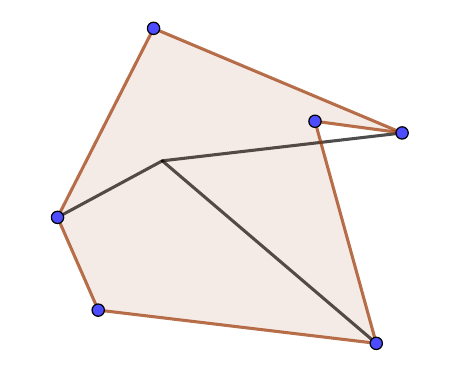}}%
    \put(0.33252824,0.79276787){\color[rgb]{0,0,0}\makebox(0,0)[t]{\lineheight{1.25}\smash{\begin{tabular}[t]{c}$A_1$\end{tabular}}}}%
    \put(0.05882095,0.34214298){\color[rgb]{0,0,0}\makebox(0,0)[t]{\lineheight{1.25}\smash{\begin{tabular}[t]{c}$A_2$\end{tabular}}}}%
    \put(0.15652544,0.10005281){\color[rgb]{0,0,0}\makebox(0,0)[t]{\lineheight{1.25}\smash{\begin{tabular}[t]{c}$A_3$\end{tabular}}}}%
    \put(0.80897442,0.00946889){\color[rgb]{0,0,0}\makebox(0,0)[t]{\lineheight{1.25}\smash{\begin{tabular}[t]{c}$A_4$\end{tabular}}}}%
    \put(0.74166683,0.50148766){\color[rgb]{0,0,0}\makebox(0,0)[t]{\lineheight{1.25}\smash{\begin{tabular}[t]{c}$A_5$\end{tabular}}}}%
    \put(0.94117901,0.5212679){\color[rgb]{0,0,0}\makebox(0,0)[t]{\lineheight{1.25}\smash{\begin{tabular}[t]{c}$A_6$\end{tabular}}}}%
    \put(0,0){\includegraphics[width=\unitlength,page=2]{alpha_realization.pdf}}%
    \put(0.33979454,0.50079057){\color[rgb]{0,0,0}\makebox(0,0)[t]{\lineheight{1.25}\smash{\begin{tabular}[t]{c}$A_7$\end{tabular}}}}%
  \end{picture}%
\endgroup%

  \caption{Here $\alpha=4$. Left: an $\alpha$-embedding of the portion of the graph
    on the left-hand side of Figure~\ref{fig:YDelta}. Right: an
    $\alpha$-realization of the portion of the graph on the right-hand side of
    Figure~\ref{fig:YDelta}. The quad $A_4A_5A_6A_7$ on the right is not proper.}
    \label{fig:embeddingrealization}
\end{figure}

As mentioned in the introduction, there are two notable classes of examples of $\alpha$-embeddings. The
class of $1$-embeddings of a planar graph corresponds to the class of $s$-embeddings
defined by Chelkak in \cite{Chelkak,Chelkak2} (see also \cite{Lis}), while the
class of $2$-embeddings corresponds to the class of harmonic embeddings
\cite{Tutte:emb,KLRR}.

\subsection{The cube move: setting}

We now define the property that we want to study on $\alpha$-quads,
which we call the \emph{flip property}. This property states that it
is possible to perform a \emph{cube move} like that of
Figure~\ref{fig:YDelta} locally on the realization or embedding, while
keeping the global structure unchanged.

\begin{definition}
\label{def:flipproperty}
For any $\alpha \in \overline{\R}$, the set of all
$\alpha$-realizations is said to satisfy the \emph{flip property} if,
for any six distinct points in the plane $A_1,A_2,\dots,A_6$ such
that $A_1,A_3,A_5$ (resp. $A_2,A_4,A_6$) are not aligned, the following are equivalent:
\begin{itemize}
\item there exists a point $A_0$ such that $A_0,A_1,\dots,A_6$ is an $\alpha$-realization of the graph on the left-hand side of Figure~\ref{fig:YDelta};
\item there exists a point $A_7$ such that $A_1,\dots,A_7$ is an $\alpha$-realization of the graph on the right-hand side of Figure~\ref{fig:YDelta}.
\end{itemize}

The set of $\alpha$-embeddings is said to satisfy the \emph{proper flip property} if, in the equivalence of Definition \ref{def:flipproperty}, it is also required that the figures are $\alpha$-embeddings, and that each of the quadrilaterals is proper, with its boundary vertices oriented in the same order as in Figure~\ref{fig:YDelta}.

In both cases, it is said to satisfy the \emph{unique} (proper) flip property if it satisfies the (proper) flip property and if, in addition, when the points $A_0$ and $A_7$ exist they are unique.
\end{definition}

Our ultimate goal is to understand for which values of $\alpha$ these
properties are satisfied by the set of $\alpha$-realizations or
$\alpha$-embeddings. In this direction, we notably prove Theorem
\ref{theo:1quad} for $1$-embeddings and Theorem~\ref{theo:>1} for
$\alpha$-realizations with $\alpha>1$, and conjecture a generalization
to other values of $\alpha$ (see Conjecture
\ref{conj:flipproperty}). Note that it is already known that the set
of $2$-embeddings satisfies the unique proper flip property and that the set of $2$-realizations satisfies the unique flip property. This can be proved using a theorem of Steiner, see
\cite{KS,KLRR}.

There is a necessary condition on the outer hexagon $A_1\ldots A_6$ in order to be in a position to flip three quads:

\begin{proposition}
\label{prop:hexagon}
Let $A_1,\dots,A_6$ be six distinct points and let $\alpha\in\R$. Suppose that there is either a point $A_0$ producing an $\alpha$-realization of the left-hand side of Figure~\ref{fig:YDelta} or a point $A_7$ producing an $\alpha$-realization of the right-hand side of Figure~\ref{fig:YDelta}. Then the side lengths of the hexagon $A_1\ldots A_6$ must satisfy
\begin{align}
A_1A_2^\alpha + A_3A_4^\alpha + A_5A_6^\alpha = A_2A_3^\alpha + A_4A_5^\alpha + A_6A_1^\alpha  &\quad \text{if } \alpha\neq0 \label{eq:alpha_hex} \\
A_1A_2. A_3A_4. A_5A_6 = A_2A_3. A_4A_5. A_6A_1 &\quad\text{if } \alpha=0. \label{eq:alpha_hex0}
\end{align}
\end{proposition}

\begin{proof}
When $\alpha\neq0$, formula~\eqref{eq:alpha_hex} follows immediately from summing the three equations defining the three $\alpha$-quads involving $A_0$ (resp. $A_7$). The case $\alpha=0$ works similarly.
\end{proof}

The example on Figure~\ref{fig:1real_no_flip} with $\alpha=1$ illustrates the fact that condition~\eqref{eq:alpha_hex} is not sufficient to have an $\alpha$-realization. Indeed, the existence of a point $A_0$ implies that condition~\eqref{eq:alpha_hex} is satisfied, but there exists no point $A_7$. We also point out that there is no analogue of Proposition~\ref{prop:hexagon} when $\alpha=\pm\infty$.

\begin{figure}[h]
  \centering
  \def\svgwidth{7cm}
  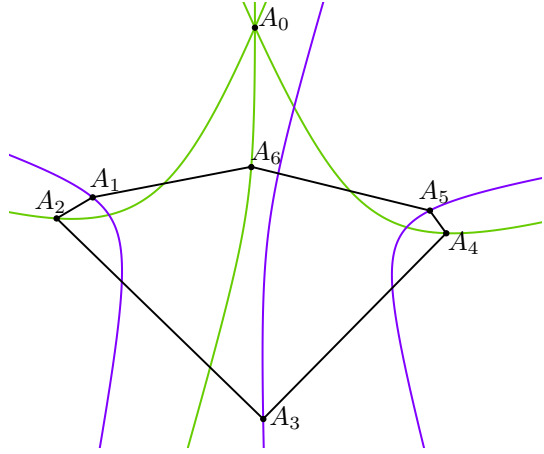
  \caption{A configuration for $\alpha=1$, with the corresponding
    construction curves. The green construction curves have an intersection point $A_0$ but the purple ones do not intersect.}
  \label{fig:1real_no_flip}
\end{figure}

\subsection{Construction curves}

In order to check the existence of the points $A_0$ and $A_7$
introduced in Definition \ref{def:flipproperty}, we will see them as intersection points of certain curves called construction curves. Let us first define them properly.

\begin{definition}
Let $A,B$ and $C$ be three distinct points in the plane and let $\alpha \in \overline \R$. The $\alpha$-\emph{construction curve} with foci $A,B$ going through $C$ is the set of points $M$ such that $ACBM$ is an $\alpha$-quad.

Let $A_1,\ldots,A_6$ be the distinct vertices of a hexagon. The $\alpha$-construction curves of the hexagon $A_1\ldots A_6$ are the six curves $\calC_i$ where for every $1\leq i\leq 6$, $\calC_i$ is defined to be the $\alpha$-construction curve with foci $A_{i-1}$ and $A_{i+1}$ going through $A_i$. Here indices are taken modulo $6$.
\end{definition}

\begin{figure}[h]
  \centering
  \def\svgwidth{12cm}
  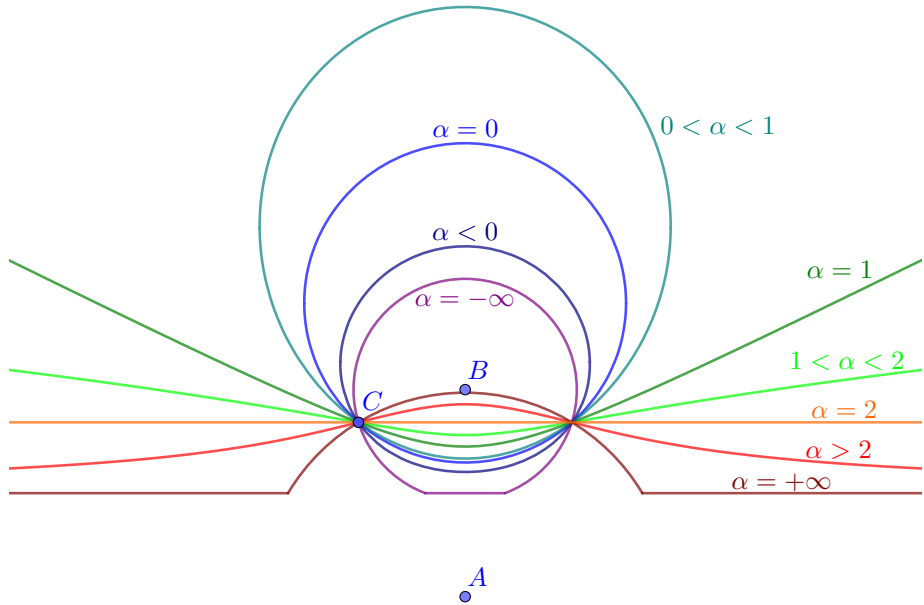
  \caption{Construction curves with foci $A$ and $B$ going through
    $C$. The fact that they intersect only at $C$ and at its symmetric
    with respect to $(AB)$ is a consequence of
    Proposition~\ref{prop:intersection}. The asymptotic
      behavior of construction curves is described in
      Lemma~\ref{lemma:asymp}.}
  \label{fig:illustr_construct_curves}
\end{figure}

The role of construction curves in our problem is the following. Start
with a hexagon $A_1A_2A_3A_4A_5A_6$ and let $\alpha\in\overline\R$. The existence of a point $A_0$ as on the left-hand side of Figure~\ref{fig:YDelta} is equivalent to the fact that $\calC_2$, $\calC_4$ and $\calC_6$ have a common point, see Figure \ref{fig:ex_intersections}. Similarly the existence of a point $A_7$ as on the right-hand side of Figure~\ref{fig:YDelta} is equivalent to the fact that $\calC_1$, $\calC_3$ and $\calC_5$ have a common point.

\begin{remark}
\label{rem:2curves}
In order for the flip property to be satisfied when $\alpha\in\R$, it is actually enough to consider the intersection of two construction curves rather than three. Indeed, if $\alpha\neq0$ and if $A_0,A_1,\dots,A_6$ are seven points such that the quadrilaterals $A_1A_2A_3A_0$, $A_3A_4A_5A_0$ and $A_5A_6A_1A_0$ are all $\alpha$-quads, then by Proposition~\ref{prop:hexagon} formula~\eqref{eq:alpha_hex} is satisfied. Hence, if $\calC_1$ and $\calC_3$ have a common point $A_7$, then combining the equation of the two $\alpha$-quads $A_2A_3A_4A_7$ and $A_6A_1A_2A_7$ with formula~\eqref{eq:alpha_hex} yields that $A_7$ also lies on $\calC_5$. The case $\alpha=0$ works similarly, replacing formula~\eqref{eq:alpha_hex} by formula~\eqref{eq:alpha_hex0}.
\end{remark}

\begin{figure}[htbp]
  \centering
  \definecolor{xfqqff}{rgb}{0.4980392156862745,0.,1.}
\definecolor{xdxdff}{rgb}{0.49019607843137253,0.49019607843137253,1.}
\definecolor{uuuuuu}{rgb}{0.26666666666666666,0.26666666666666666,0.26666666666666666}
\definecolor{wwccqq}{rgb}{0.4,0.8,0.}
\begin{tikzpicture}[line cap=round,line join=round,x=1.0cm,y=1.0cm, scale=2]
\clip(-0.982883474532752,-1.4670062146945726) rectangle (1.840114012870508,0.932859636113895);
\draw [line width=1.pt,color=wwccqq] (0.5838341505140906,0.4376643694032091) circle (0.5441005951115861cm);
\draw [line width=1.pt,color=wwccqq] (-3.7219161778991188,4.717896557368066) circle (6.60019283989575cm);
\draw [line width=1.pt,color=wwccqq] (1.603132952483884,-0.575593531909651) circle (0.9192761058052245cm);
\draw [line width=1.pt,color=xfqqff] (-0.08890167820051813,0.11838511401555643) circle (0.5324426891464713cm);
\draw [line width=1.pt,color=xfqqff] (-0.17160406707526732,0.6515193378698584) circle (1.0896766753476703cm);
\draw [line width=1.pt,color=xfqqff] (0.40525287922990966,-3.0671420518167003) circle (2.6483128162262393cm);
\draw [line width=1.pt] (0.41110554233457375,0.3013823725696064)-- (0.11812284798744888,0.15631327148510765);
\draw [line width=1.pt] (0.11812284798744888,0.15631327148510765)-- (-0.4725112398056511,-0.39578679225454766);
\draw [line width=1.pt] (-0.4725112398056511,-0.39578679225454766)-- (0.8735081792599249,-1.1348039118225282);
\draw [line width=1.pt] (0.8735081792599249,-1.1348039118225282)-- (1.1990298756758708,-0.5405873121561119);
\draw [line width=1.pt] (1.1990298756758708,-0.5405873121561119)-- (1.2360082739915272,0.36111670831028236);
\draw [line width=1.pt] (1.2360082739915272,0.36111670831028236)-- (0.41110554233457375,0.3013823725696064);
\begin{scriptsize}
\draw [fill=black] (0.41110554233457375,0.3013823725696064) circle (0.8pt);
\draw[color=black] (0.4310169875814654,0.3540054778649638) node [above] {$A_1$};
\draw [fill=black] (0.11812284798744888,0.15631327148510765) circle (0.8pt);
\draw[color=black] (0.13803429323434052,0.20893637678046503) node [above] {$A_2$};
\draw [fill=black] (-0.4725112398056511,-0.39578679225454766) circle (0.8pt);
\draw[color=black] (-0.453620079816164,-0.34289510577625565) node [left] {$A_3$};
\draw [fill=black] (1.1990298756758708,-0.5405873121561119) circle (0.8pt);
\draw[color=black] (1.2189413209227622,-0.4879642068607544) node [right] {$A_5$};
\draw [fill=black] (1.2360082739915272,0.36111670831028236) circle (0.8pt);
\draw[color=black] (1.2559197192384188,0.4137398136056397) node [above] {$A_6$};
\draw [fill=uuuuuu] (1.051717107790615,0.15993969298615707) circle (0.8pt);
\draw[color=uuuuuu] (1.071027727660136,0.2460918846023376) node {$A_0$};
\draw [fill=black] (0.8735081792599249,-1.1348039118225282) circle (0.8pt);
\draw[color=black] (0.8918247204381083,-1.1481520564456413) node [below] {$A_4$};
\draw [fill=uuuuuu] (0.8643350577483865,-0.028559502402513863) circle (0.8pt);
\draw[color=uuuuuu] (0.883291243903726,0.058355400845927475) node {$A'_0$};
\end{scriptsize}
\end{tikzpicture}
  \definecolor{xfqqff}{rgb}{0.4980392156862745,0.,1.}
\definecolor{xdxdff}{rgb}{0.49019607843137253,0.49019607843137253,1.}
\definecolor{uuuuuu}{rgb}{0.26666666666666666,0.26666666666666666,0.26666666666666666}
\definecolor{wwccqq}{rgb}{0.4,0.8,0.}
\definecolor{ududff}{rgb}{0.30196078431372547,0.30196078431372547,1.}
\begin{tikzpicture}[line cap=round,line join=round,x=1.0cm,y=1.0cm, scale=2]
\clip(-0.982883474532752,-1.4670062146945726) rectangle (1.840114012870508,0.932859636113895);
\draw [line width=1.pt,color=wwccqq] (0.5838341505140906,0.4376643694032091) circle (0.5441005951115861cm);
\draw [line width=1.pt,color=wwccqq] (5.525353763465017,-4.919412686501057) circle (6.803122475564672cm);
\draw [line width=1.pt,color=wwccqq] (1.3523444646763976,-0.39547387423236413) circle (0.6869033148582944cm);
\draw [line width=1.pt,color=xfqqff] (-0.08890167820051813,0.11838511401555643) circle (0.5324426891464713cm);
\draw [line width=1.pt,color=xfqqff] (-0.2707448431274471,0.703549611186293) circle (1.1176986207197133cm);
\draw [line width=1.pt,color=xfqqff] (0.47948549127446516,-1.710663518952187) circle (1.4515403636094448cm);
\draw [line width=1.pt] (0.41110554233457375,0.3013823725696064)-- (0.11812284798744888,0.15631327148510765);
\draw [line width=1.pt] (0.11812284798744888,0.15631327148510765)-- (-0.4725112398056511,-0.39578679225454766);
\draw [line width=1.pt] (-0.4725112398056511,-0.39578679225454766)-- (0.807152794659839,-0.8133281633187261);
\draw [line width=1.pt] (0.807152794659839,-0.8133281633187261)-- (1.093783665085156,-0.3955182110716131);
\draw [line width=1.pt] (1.093783665085156,-0.3955182110716131)-- (1.2360082739915272,0.36111670831028236);
\draw [line width=1.pt] (1.2360082739915272,0.36111670831028236)-- (0.41110554233457375,0.3013823725696064);
\begin{scriptsize}
\draw [fill=black] (0.41110554233457375,0.3013823725696064) circle (0.8pt);
\draw[color=black] (0.4310169875814654,0.3540054778649638) node [above] {$A_1$};
\draw [fill=black] (0.11812284798744888,0.15631327148510765) circle (0.8pt);
\draw[color=black] (0.13803429323434052,0.20893637678046503) node [above] {$A_2$};
\draw [fill=black] (-0.4725112398056511,-0.39578679225454766) circle (0.8pt);
\draw[color=black] (-0.453620079816164,-0.34289510577625565) node [left] {$A_3$};
\draw [fill=black] (1.093783665085156,-0.3955182110716131) circle (0.8pt);
\draw[color=black] (1.113695110332048,-0.34289510577625565) node [right] {$A_5$};
\draw [fill=black] (1.2360082739915272,0.36111670831028236) circle (0.8pt);
\draw[color=black] (1.2559197192384188,0.4137398136056397) node [above] {$A_6$};
\draw [fill=uuuuuu] (1.0906077316950136,0.23960883421712942) circle (0.8pt);
\draw[color=uuuuuu] (1.1108506181539206,0.30573766558990554) node {$A_0$};
\draw [fill=black] (0.807152794659839,-0.8133281633187261) circle (0.8pt);
\draw[color=black] (0.8235569081630506,-0.8267244403172422) node [below] {$A_4$};
\draw [fill=uuuuuu] (0.270774818983947,-0.2742063206829597) circle (0.8pt);
\draw[color=uuuuuu] (0.2616368708532219,-0.16627092647303115) node {$A_7$};
\draw [fill=uuuuuu] (0.7404147490581201,-0.0834190920119231) circle (0.8pt);
\draw[color=uuuuuu] (0.7609780802442472,-0.00568995053849362) node {$A'_0$};
\draw [fill=uuuuuu] (-0.16271633853162257,-0.4089161406252129) circle (0.8pt);
\draw[color=uuuuuu] (-0.16063738546903888,-0.3081413163669704) node {$A'_7$};
\end{scriptsize}
\end{tikzpicture}
  \def\svgwidth{11cm}
  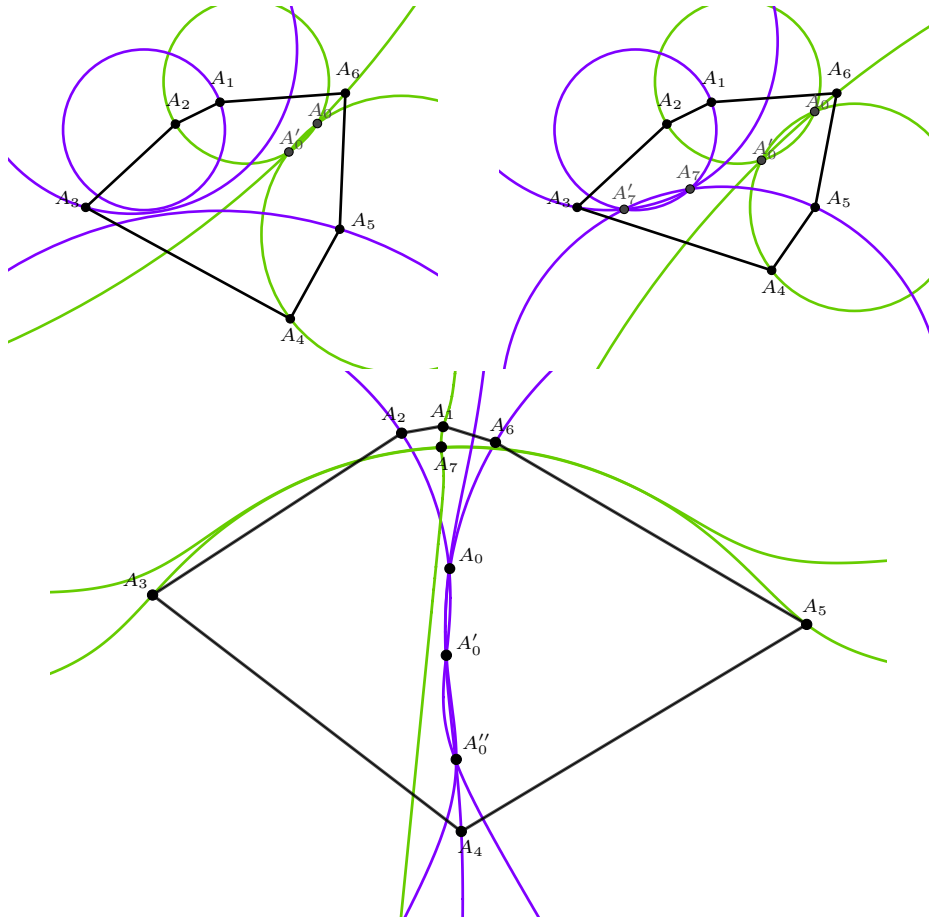
  \caption{Examples of construction curves for some hexagons. Top
    left: for $\alpha=0$, two possible points $A_0$ and no point
    $A_7$. Top right: for $\alpha=0$, two points $A_0$ and two points
    $A_7$. Pictures with the same number of solutions as the top two pictures can be obtained for $\alpha$ close to $0$. Bottom: for $\alpha=10$, three points $A_0$ and
    one point $A_7$.}
  \label{fig:ex_intersections}
\end{figure}

In a few cases, the construction curves are well known. Fix $A$ and $B$ two points in the plane.
If $C$ is a point such that $AC=BC$, then clearly for any $\alpha$ the construction curve going through $C$ is the perpendicular bisector of $[AB]$, which we denote from now on by $P(AB)$. Without loss of generality, we may assume in what follows that $BC < AC$. See Figure \ref{fig:illustr_construct_curves} for the general shape of construction curves depending on the value of $\alpha$.

For any point $O$, we denote by $\calC(O,C)$ the circle with center $O$ going through $C$, and $\calD(O,C)$ the closed disk enclosed by $\calC(O,C)$. The following proposition characterizes the construction curves in the special cases $\alpha \in \{-\infty, 0,1,2,+\infty \}$.

\begin{proposition}
  \label{prop:special_construction}
  Let $A,B$ and $C$ be three points such that $BC<AC$. We denote by $H$ the closed half-plane bounded by $P(AB)$ containing $B$. Then one has the following description for the $\alpha$-construction curve with foci $A,B$ going through $C$:
\begin{itemize}
\item For $\alpha=-\infty$, it is the union of the circular arc $\calC(B,C)\cap H$ and the segment $P(A,B) \cap \calD(B,C)$;
\item For $\alpha=+\infty$, it is the union of the circular arc $\calC(A,C)\cap H$ and the two half-lines $P(A,B) \cap \calD(A,C)^c$;
\item For $\alpha=0$, it is a circle going through $C$;
\item For $\alpha=1$, it is the branch closest to $B$ of the hyperbola with foci $A,B$ going through $C$;
\item For $\alpha=2$, it is the perpendicular to $(AB)$ going through $C$.
\end{itemize}
\end{proposition}

\begin{proof}
For $\alpha=-\infty$, the curve is $\{M \in \R^2 \mid \min(AM,BC) = \min(BM,AC)\}$. The explicit description follows from a simple case handling, depending of the length realizing the minimum. The same goes for $\alpha=+\infty$.

For $\alpha = 0$, the curve is the set of points $M$ such that the ratio $\frac{MA}{MB}$ is fixed, which is a circle by Apollonius's circle theorem.

For $\alpha = 1$, the curve is the set of points $M$ such that $AM-BM$ is a fixed positive number, which is a branch of hyperbola with foci $A,B$.

For $\alpha = 2$, as stated in Section~\ref{sec:space}, the quadrilateral $ACBM$ is a $2$-quad if and only if $(AB)$ and $(CM)$ are perpendicular.
\end{proof}

\section{The cube move for $\alpha=1$}
\label{sec:Ising}

Our goal in this section is to prove the unique proper flip property for $1$-embeddings.

\begin{theorem}
  \label{theo:1quad}
The set of $1$-embeddings satisfies the unique proper flip property.
\end{theorem}

Although our proof only works in this specific case, based on the observation of many numerical examples we expect the following more general result to hold:
\begin{conjecture}
\label{conj:flipproperty}
For any $\alpha \in [-\infty, 1]$, for any six distinct points in the
plane $A_1,A_2,\dots,A_6$ such
that $A_1,A_3,A_5$ (resp. $A_2,A_4,A_6$) are not aligned, the following are equivalent:
  \begin{itemize}
  \item there exists a \emph{unique} point $A_0$ such that $A_0,A_1,\dots,A_6$ is an $\alpha$-\emph{embedding} of the graph on the left-hand side of Figure~\ref{fig:YDelta};
  \item there exists a \emph{unique} point $A_7$ such that $A_1,\dots,A_7$ is an $\alpha$-\emph{embedding} of the graph on the right-hand side of Figure~\ref{fig:YDelta}.
  \end{itemize}
\end{conjecture}
Notice that this is weaker than the unique proper flip property, as it
is possible that several points $A_0$ exist and no point $A_7$. Such a
configuration is shown in the top-left picture of
Figure~\ref{fig:ex_intersections}. For the sake of
completeness, let us also mention that the set of $1$-realizations does
not satisfy the flip property (when there is no ``proper''
requirement): there exist configurations such that there is a point
$A_0$ yielding $1$-realizations, but there is no point $A_7$ (see Figure~\ref{fig:1real_no_flip}).

The rest of this section is devoted to the proof of Theorem~\ref{theo:1quad}.

\subsection{Uniqueness}

In this first part, we prove the following proposition, which states the uniqueness of the point $A_0$, if it exists.
\begin{proposition}
  \label{prop:uniqueness}
For any proper positively oriented hexagon $A_1,A_2,\dots,A_6$, there exists at most one point $A_0$ such that $A_0,A_1,\dots,A_6$ is a proper $1$-embedding of the left-hand side of Figure~\ref{fig:YDelta}.
\end{proposition}

In order to prove it, we first need some information on the geometric properties of $1$-quads. Notice that if three distinct points $A,B,C$ are fixed, then the construction curve for $1$-quads is the set of points $D$ such that
\begin{equation}
AD-CD = AB - BC.
\end{equation}
Hence, it is a hyperbola branch with foci $A,C$ (with possible degenerate cases being the perpendicular bisector of $[AC]$, and
half-lines $A + t (A-C)$ or $C + t(C-A)$ for $t\geq 0$). We put all these cases under the same name:
\begin{definition}
  Let $A,C$ be two distinct points in the plane. For any $\lambda \in
  \R$, the set of points $D$ in the plane such that
  \begin{equation}
    \label{eq:gen_br}
    AD - CD = \lambda
  \end{equation}
  is called a \emph{generalised hyperbola branch} with foci $A$ and $C$.
\end{definition}

The following lemma already implies that there are at most two admissible points $A_0$, in the sense of Proposition \ref{prop:uniqueness}.
\begin{lemma}
\label{lemma:branches_intersection}
Assume that two generalised hyperbola branches have exactly one common focus, then they have at most two intersection points.
\end{lemma}

\begin{figure}[h]
  \centering
  \def\svgwidth{10cm}
\begingroup%
  \makeatletter%
  \providecommand\color[2][]{%
    \errmessage{(Inkscape) Color is used for the text in Inkscape, but the package 'color.sty' is not loaded}%
    \renewcommand\color[2][]{}%
  }%
  \providecommand\transparent[1]{%
    \errmessage{(Inkscape) Transparency is used (non-zero) for the text in Inkscape, but the package 'transparent.sty' is not loaded}%
    \renewcommand\transparent[1]{}%
  }%
  \providecommand\rotatebox[2]{#2}%
  \newcommand*\fsize{\dimexpr\f@size pt\relax}%
  \newcommand*\lineheight[1]{\fontsize{\fsize}{#1\fsize}\selectfont}%
  \ifx\svgwidth\undefined%
    \setlength{\unitlength}{270.1715965bp}%
    \ifx\svgscale\undefined%
      \relax%
    \else%
      \setlength{\unitlength}{\unitlength * \real{\svgscale}}%
    \fi%
  \else%
    \setlength{\unitlength}{\svgwidth}%
  \fi%
  \global\let\svgwidth\undefined%
  \global\let\svgscale\undefined%
  \makeatother%
  \begin{picture}(1,0.52130071)%
    \lineheight{1}%
    \setlength\tabcolsep{0pt}%
    \put(0,0){\includegraphics[width=\unitlength,page=1]{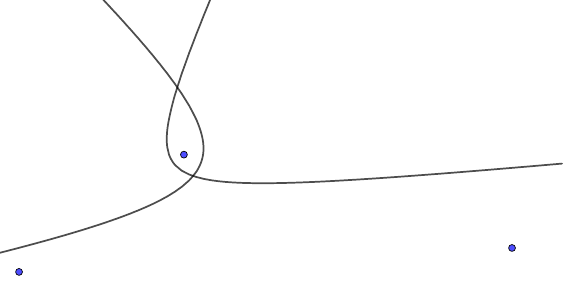}}%
    \put(0.30698824,0.26590961){\color[rgb]{0,0,0}\makebox(0,0)[lt]{\lineheight{1.25}\smash{\begin{tabular}[t]{l}$A_1$\end{tabular}}}}%
    \put(0.02990486,0.05224445){\color[rgb]{0,0,0}\makebox(0,0)[lt]{\lineheight{1.25}\smash{\begin{tabular}[t]{l}$A_3$\end{tabular}}}}%
    \put(0.90097331,0.09680948){\color[rgb]{0,0,0}\makebox(0,0)[lt]{\lineheight{1.25}\smash{\begin{tabular}[t]{l}$A_5$\end{tabular}}}}%
  \end{picture}%
\endgroup%

  \caption{A case of two branches intersecting at two points. The focus
    $A_1$ is the interior focus of both branches.}
  \label{fig:uniqueness}
\end{figure}
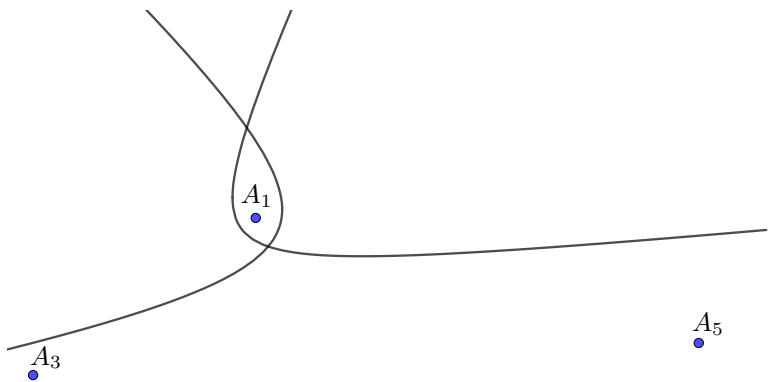

Although this result, which appears in \cite{LSFW,XSR}, has a very classical flavor, we could not find any earlier reference. We give an alternative proof below.

Let us start with some definitions. The focus of a hyperbola branch $\calB$ which belongs to (resp. does not belong to) the convex hull of $\calB$ is said to be the \emph{interior} (resp. \emph{exterior})
focus of $\calB$. For example, on Figure~\ref{fig:uniqueness}, $A_1$ is the interior focus of both branches while $A_3$ and $A_5$ are the exterior foci of these branches. Assume that $\calB_1$ is a hyperbola branch with foci $A$ and $C$ and $\calB_2$ is a hyperbola branch with foci $A$ and $E$, with $A,C,E$ distinct. Then any intersection point of $\calB_1$ and $\calB_2$ lies on a hyperbola branch $\calB_3$ with foci $C$ and $E$ (whose equation is obtained by subtracting the equations \eqref{eq:gen_br} for $\calB_1$ and $\calB_2$), and we can also suppose that $\calB_3$ is not a line or a half-line. Then there is at least one of the three points $A$, $C$ or $E$ which is the interior focus of one branch and the exterior focus of another branch. Without loss of generality, we assume that $A$ is the interior focus of $\calB_1$ and the exterior focus of $\calB_2$. We will show that these two branches intersect in at most two points.

Suppose that $\calB_1$ and $\calB_2$ intersect at three distinct points $S,T,U$. As these points belong to a non-degenerate hyperbola branch, they are not aligned. The lines $(ST),(TU),(SU)$ therefore delimit exactly seven open regions in the plane, three of which touch the triangle $STU$ only at one vertex. We call these three regions the \emph{corner chambers} of $S,T,U$.

\begin{lemma}
\label{lemma:chamber}
Let $S,T$ and $U$ be three distinct points on a hyperbola branch $\calB$. Then the exterior focus of $\calB$ belongs to a corner chamber of $S,T,U$, while the interior focus does not.
\end{lemma}

\begin{proof}[Proof of Lemma \ref{lemma:branches_intersection}]
If one of the generalised branches is a line or a half-line, the result easily comes from the fact that a hyperbola and a line have at most two points of intersection.

Suppose now that it is not the case. The proof of Lemma~\ref{lemma:branches_intersection} follows from Lemma \ref{lemma:chamber}, as $A$ should be both in a corner chamber of $S,T,U$ (because $A$ is the exterior focus of $\calB_2$) and not in one (because $A$ is the interior focus of $\calB_1$).
\end{proof}

\begin{proof}[Proof of Lemma~\ref{lemma:chamber}]
For any two distinct points $A,B$ on $\calB$, the line $(AB)$ cuts the interior of the convex hull of $\calB$ into a finite part and an infinite part. The half-plane delimited by $(AB)$ that contains the finite part is called the \emph{exterior half-plane} of $A,B$. It is easy to see that the exterior focus belongs to the exterior half-plane of $A,B$, for instance by noting that this property is invariant by affine transformations of the plane, and is straightforward to prove for the special branch $\{(x,y) \in (0,\infty)^2 \mid xy=1\}$.

Suppose that $T$ belongs to the exterior half-plane of $S,U$ (which is equivalent to saying that $S,T,U$ are met in that order when following the branch $\calB$). Then, applying the previous property to $(S,T)$ and $(T,U)$, we get that the exterior focus has to belong to the corner chamber that touches $T$.

We now consider the interior focus of $\calB$. By convexity of $\calB$, the corner chambers are disjoint from the interior of the convex hull of $\calB$, which by definition contains its interior focus. The result follows.
\end{proof}

We can now prove Proposition \ref{prop:uniqueness}.

\begin{proof}[Proof of Proposition~\ref{prop:uniqueness}]
  By Lemma~\ref{lemma:branches_intersection}, there exists at most two
  such points. Suppose that there are two, $A_0$ and $A'_0$. We claim
  that $A_1,A_3,A_5$ belong to a unique hyperbola branch with foci
  $A_0,A'_0$. Indeed, since $A_0,A'_0$ are on the same hyperbola
  branch with foci $A_1,A_3$, we have
  $A_1A_0 - A_3A_0 = A_1A'_0 - A_3A'_0$. This is equivalent to
  $A_1A_0 - A_1A'_0=A_3A_0-A_3A'_0$, which in turn means that
  $A_1,A_3$ are on the same hyperbola branch with foci $A_0,A'_0$. The
  same holds for $A_5$.

  Therefore, by Lemma~\ref{lemma:chamber}, one of the points
  $A_0,A'_0$ is in a corner chamber of the triangle $A_1A_3A_5$, and
  the other is not. But notice that for any point $M$ that does not belong to the lines
  $(A_1A_3),(A_1A_5),(A_3A_5)$, the cyclic order of the vectors
  $\overrightarrow{MA_1}, \overrightarrow{MA_3},
  \overrightarrow{MA_5}$
  is always the same when $M$ belongs to the union
  of the three corner chambers,
  and the opposite when $M$ is
  outside that union.
However, the cyclic order of the vectors
 $\overrightarrow{A_0A_1}, \overrightarrow{A_0A_3},
 \overrightarrow{A_0A_5}$ should be the same as the cyclic order of the vectors
 $\overrightarrow{A'_0A_1}, \overrightarrow{A'_0A_3},
 \overrightarrow{A'_0A_5}$, since this order should be fixed by the
 proper embedding. Therefore it is impossible for both $A_0$ and
 $A'_0$ to correspond to proper embeddings.
\end{proof}

\subsection{Existence}

The second part of this section consists in proving that, if we start
with a proper $1$-embedding $A_0,A_1,\dots,A_6$ of the left-hand side
of Figure~\ref{fig:YDelta}, then there actually exists a point $A_7$
inducing a proper $1$-embedding of the right-hand side. To show this,
we transform the problem into a linear one by using the
\emph{propagation equations} and \emph{$s$-embeddings} defined by
Chelkak \cite{Chelkak,Chelkak2}. We briefly explain this construction (we refer
to Chelkak's original papers for more details), then give a few
extra properties concerning the ordering of the vertices of the quads
it gives, and finally apply it to our setting.

\subsubsection{Ising model, propagation equations and $s$-embeddings}
\label{sec:ising-model-prop}

Let $G := (V,E)$ be a finite planar graph, in which each edge $e\in E$ carries a positive weight $J_e>0$. The weights $(J_e)_{e\in E}$ are called the coupling constants of the ferromagnetic Ising model on $G$, that is, every spin configuration $\sigma \in \{\pm 1\}^V$ is assigned the Boltzmann weight
\begin{equation}
  w(\sigma) = \exp\left( \sum_{e = \{u,v\} \in E} J_e \sigma_u \sigma_v
  \right)
\end{equation}
and a spin configuration is randomly sampled with probability proportional to its Boltzmann weight. Note however that we will not refer to any statistical mechanical property of the Ising model thereafter; all the proofs are of purely geometric nature.

One checks that, for every $e\in E$ there exists a unique $\theta_e \in (0,\frac{\pi}{2})$ such that
\begin{equation}
\label{eq:defJ}
  J_e = \frac12 \ln\left(\frac{1+\sin\theta_e}{\cos\theta_e}\right).
\end{equation}
We also set
\begin{equation}
\label{eq:defx}
  x_e = \tanh J_e = \tan \frac{\theta_e}{2} \in (0,1).
\end{equation}

\medskip

Let $G^c$ be the weighted graph whose vertices are the \emph{corners} of the faces of $G$, and whose edges are of two types:
\begin{enumerate}
 \item those connecting two corners that correspond to the same vertex and to the same edge $e$ of $G$; such edges carry the weight $\cos\theta_e$;
 \item those connecting two corners that correspond to the same face and to the same edge $e$ of $G$; such edges carry the weight $\sin\theta_e$.
\end{enumerate}
See Figure~\ref{fig:gc} for an illustration. There exists a double cover $\dbc$ of $G^c$ that branches around every edge, vertex and face of $G$, graphically represented around an edge in Figure~\ref{fig:propagation} (see also Figure~\ref{fig:st_propagation}), which inherits the edge weights of $G^c$.

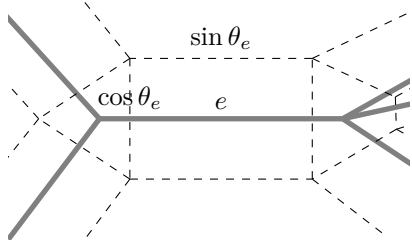
\begin{figure}[h]
  \centering
  \begin{tikzpicture}[scale=0.8]
\clip(-7,-3) rectangle (-0.4,1);
\draw [line width=2pt,gray] (-5.5,-1)-- (-1.5,-1);
\draw [line width=2pt,gray] (-5.5,-1)-- (-7.32,1.02);
\draw [line width=2pt,gray] (-5.5,-1)-- (-7,-3);
\draw [line width=2pt,gray] (-1.5,-1)-- (0.22962035925127702,-0.016947265540150244);
\draw [line width=2pt,gray] (-1.5,-1)-- (0.07419079774291935,-0.6559354628522881);
\draw [line width=2pt,gray] (-1.5,-1)-- (0,-2);
\draw [dashed] (-5,0)-- (-5,-2);
\draw [dashed] (-5,-2)-- (-2,-2);
\draw [dashed] (-2,-2)-- (-2,0);
\draw [dashed] (-2,0)-- (-5,0);
\draw [dashed] (-5,0)-- (-6.5,-1);
\draw [dashed] (-5,-2)-- (-6.5,-1);
\draw [dashed] (-6.5,-1)-- (-7.783637034068496,0.4407064433455701);
\draw [dashed] (-5,0)-- (-6.117086735673328,1.399188739313777);
\draw [dashed] (-6.5,-1)-- (-7.533222740527253,-2.305215809968752);
\draw [dashed] (-5,-2)-- (-5.823497563935319,-3.1082685444286007);
\draw [dashed] (-2,0)-- (-0.6338772046840433,-0.6386655115735816);
\draw [dashed] (-0.6338772046840433,-0.6386655115735816)-- (-0.616607253405337,-1.174034001213481);
\draw [dashed] (-0.616607253405337,-1.174034001213481)-- (-2,-2);
\draw [dashed] (-2,-2)-- (-0.8411166200285203,-2.9010291290841237);
\draw [dashed] (-2,0)-- (-0.279843203470562,0.8206453714771115);
\draw [dashed] (-0.6338772046840433,-0.6386655115735816)-- (0.5,0);
\draw [dashed] (-0.616607253405337,-1.174034001213481)-- (0.7736238245305288,-0.733650243606467);
\draw [dashed] (-0.616607253405337,-1.174034001213481)--
(0.8081637270879417,-1.9684517600339766);
\draw (-3.5,-1) node [above] {$e$};
\draw (-3.5,0) node [above] {$\sin \theta_e$};
\draw (-5,-1) node [above] {$\cos \theta_e$};
\end{tikzpicture}
  \caption{The corner graph $G^c$ (dashed) around an edge $e$ of $G$
    (solid).}
  \label{fig:gc}
\end{figure}

Let $V^{\times}$ be the vertices of $\dbc$. We say that a function $X: V^{\times} \to \C$ satisfies the \emph{propagation equation} if, for every $v \in V^{\times}$ with neighbours $v',v'' \in V^{\times}$ around an edge $e$ like in Figure~\ref{fig:propagation},
\begin{equation}
  X_v = \sin\theta_e X_{v'} + \cos\theta_e X_{v''}.
\end{equation}
It is easy to check that if $X$ satisfies the propagation equation, its value is multiplied by $-1$ whenever we change sheets above a vertex of $G^c$.
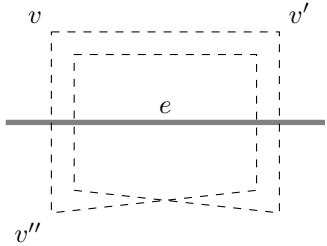
\begin{figure}[h]
  \centering
  \begin{tikzpicture}[scale=0.6]
\draw [line width=2pt,gray] (0,0) -- (7,0);
\draw [dashed] (1,2) -- (6,2) -- (6,-2) -- (1.5,-1.5) -- (1.5,1.5) --
(5.5,1.5) -- (5.5,-1.5) -- (1,-2) -- cycle;
\draw (1,2) node [above left] {$v$};
\draw (6,2) node [above right] {$v'$};
\draw (1,-2) node [below left] {$v''$};
\draw (3.5,0) node [above] {$e$};
\end{tikzpicture}
  \caption{The double cover $\dbc$ around the edge $e$ of $G$.}
  \label{fig:propagation}
\end{figure}

\medskip

If $X$ is a solution to the propagation equation, then one can
construct a function $\mathcal{S} : G^{\diamond} \to \C$,
which is called the \emph{$s$-embedding} associated to $X$, in the following way.

Fix the image $\mathcal{S}(u_0)$ of a base vertex $u_0$ of $G$ in the plane. Then define $\mathcal{S}$ such that for every vertex $u$ of $G$ and every face $f$ adjacent to it, denoting by $c$ the corner between $u$ and $f$,
\begin{equation}
  \label{eq:incrS}
 \mathcal{S}(f) - \mathcal{S}(u) = X_c^2
\end{equation}
where $X_c$ is any of the two values of $X$ above the corner $c$;
both values produce the same constraint on $\mathcal{S}$. See
Figure~\ref{fig:prop_loc} for an example. Notice that it is not
clear \emph{a priori} that $\mathcal{S}$ is well-defined, as one needs to check at
least that conditions~\eqref{eq:incrS} are closed around an edge, as
in Figure~\ref{fig:prop_loc}. We
will rely on the following result of Chelkak \cite{Chelkak} that
asserts that the $s$-embedding $\mathcal{S}$ is well-defined, and
identifies $s$-embeddings with our notion of $1$-embedding.

\begin{proposition}[\cite{Chelkak}]
\label{prop:chelkak}
For any solution $X$ of the propagation equation such that
$\Re(X),\Im(X)$ are two vectors independent over $\R$, the associated
$s$-embedding $\mathcal{S}$ is well-defined, and is such that every face of $G^{\diamond}$ is sent to a proper $1$-quad in the complex plane. Conversely, for any $1$-embedding $\mathcal{T}$ of $G$, for any edge $e\in E$ let $\theta_e$ be the unique angle in $\left(0,\frac{\pi}{2}\right)$ such that, using the notation of Figure~\ref{fig:1quad},
\begin{equation}
    \tan^2\theta_e = \frac{\cotan \delta + \cotan \beta}{\cotan \alpha
      + \cotan \gamma}.
  \end{equation}
Then $\mathcal{T}$ is an $s$-embedding associated to a solution of the propagation equations on $G$ with parameters $(\theta_e)_{e\in E}$.
\end{proposition}

The plan of our proof of the existence part of the flip property is therefore to translate the initial configuration into a solution of the propagation
equations, then show that one can apply a star-triangle transformation on this solution, and finally go back to embeddings. However, we need more information than what is contained in Proposition~\ref{prop:chelkak}, as we want to keep track of the
orientation of quadrilaterals. This is the aim of the following subsection.

\subsubsection{Orientation of quads in $s$-embeddings}

\begin{figure}[h]
  \centering
  \begin{tikzpicture}[scale=0.7]
\draw [line width=2pt,gray] (0,0) -- (7,0);
\draw [line width=2pt,gray, dashed] (3.5,-2.5) -- (3.5,2.5);
\draw [dashed] (1,2) -- (6,2) -- (6,-2) -- (1.5,-1.5) -- (1.5,1.5) --
(5.5,1.5) -- (5.5,-1.5) -- (1,-2) -- cycle;
\draw (1,2) node [above left] {$d$};
\draw (1.2,1.8) node [below right] {$-d$};
\draw (6,2) node [above right] {$c$};
\draw (5.8,1.8) node [below left] {$-c$};
\draw (1,-2) node [below left] {$a$};
\draw (1.2,-1.8) node [above right] {$-a$};
\draw (6,-2) node [below right] {$b$};
\draw (5.8,-1.8) node [above left] {$-b$};
\draw (3.5,0) node [above left] {$e$};
\draw (0,0) node [left] {$u$};
\draw (7,0) node [right] {$v$};
\draw (3.5,-2.5) node [below] {$x$};
\draw (3.5,2.5) node [above] {$y$};

\draw (9,0) node [] {$\longrightarrow$};
\node [draw=black,fill=black, thick, circle, inner sep=0pt, minimum
size=3pt] at (12,0) {};
\node [draw=black, fill=black, thick, circle, inner sep=0pt, minimum
size=3pt] at (16,0.5) {};
\node [draw=black, fill=white, thick, circle, inner sep=0pt, minimum
size=3pt] at (14,2) {};
\node [draw=black, fill=white, thick, circle, inner sep=0pt, minimum
size=3pt] at (14.8,-1.9) {};

\draw [-{Latex[length=3mm]}] (12,0) -- (14,2);
\draw [-{Latex[length=3mm]}] (12,0) -- (14.8,-1.9);
\draw [-{Latex[length=3mm]}] (16,0.5) -- (14,2);
\draw [-{Latex[length=3mm]}] (16,0.5) -- (14.8,-1.9);

\draw (12,0) node [left] {$\mathcal{S}(u)$};
\draw (16,0.5) node [right] {$\mathcal{S}(v)$};
\draw (14,2) node [above] {$\mathcal{S}(y)$};
\draw (14.8,-1.9) node [below] {$\mathcal{S}(x)$};

\draw (13,1) node [left] {$d^2$};
\draw (15,1.25) node [above right] {$c^2$};
\draw (13.4,-0.9) node [below left] {$a^2$};
\draw (15.4,-0.7) node [right] {$b^2$};

\end{tikzpicture}

  \caption{An edge $e\in E$ with vertices $u,v$ and adjacent faces $x,y$, with a solution of its propagation equation, and the corresponding $s$-embedding, where the numbers $a^2,\ldots,d^2$ are the complex coordinates of the vectors bounding the quad.}
  \label{fig:prop_loc}
\end{figure}
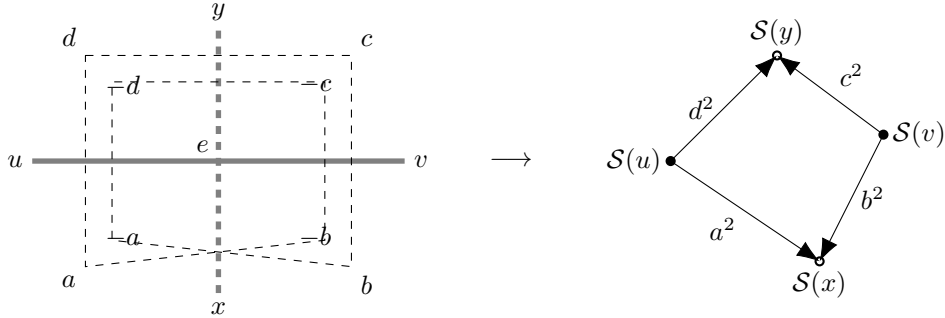

\begin{lemma}
  \label{lemma:orient}
  Let $a,b,c,d \in \C^*$ be a solution to the propagation equation at an edge $e\in E$ with parameter $\theta \in \left(0,\frac{\pi}{2}\right)$, set around $e$ as in Figure~\ref{fig:prop_loc}. Suppose that $b/a \notin \R$. Then the following are equivalent:
  \begin{enumerate}[label=(\roman*)]
  \item The $1$-quad $\mathcal{S}(u) \mathcal{S}(x) \mathcal{S}(v) \mathcal{S}(y)$ is positively oriented;
  \item $\Im(b/a) > 0$;
  \item The cyclic order around the circle of the arguments of $\pm a, \pm b, \pm c, \pm d$ is $(a, \, d, \,c, \,b, \,-a, \,-d, \,-c, \,-b)$ (see Figure~\ref{fig:order}).
  \end{enumerate}
\end{lemma}

\begin{figure}[h]
  \centering
  \def\svgwidth{4cm}
  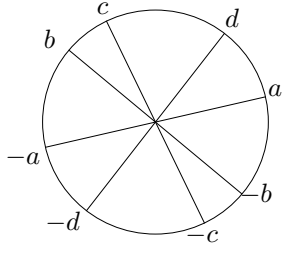
  \caption{Cyclic order of the arguments of the complex numbers
    $a,b,c,d$.}
  \label{fig:order}
\end{figure}

\begin{proof}
The three statements are unaltered if we multiply all the complex numbers $a,b,c,d$ by the same nonzero complex number. Hence we can suppose $a=1$.
Since the quad $\mathcal{S}(u) \mathcal{S}(x) \mathcal{S}(v) \mathcal{S}(y)$ is proper by Proposition~\ref{prop:chelkak}, the sum of the representatives in $(0,2\pi)$ of the oriented angles between pairs of vectors $\widehat{(a^2,d^2)}, \widehat{(b^2,a^2)}, \widehat{(c^2,b^2)}, \widehat{(d^2,c^2)}$ has to be either $2\pi$ if the quad is oriented positively or $6\pi$ if the quad is oriented negatively. In particular, the quad is oriented positively (resp. negatively) if and only if the arguments of $a^2,d^2,c^2,b^2$ (resp. $b^2,c^2,d^2,a^2$) are in that cyclic order around the circle.
  \begin{itemize}
    \item $(iii) \Rightarrow (i)$: Immediate consequence of the previous sentence.
    \item $(ii) \Rightarrow (iii)$:
    Suppose that $\Im(b)>0$. Solving the propagation equation gives
    \begin{equation}
      \begin{split}
        c &= \frac{b}{\cos \theta} + \tan \theta, \\
        d &= \frac{1}{\cos  \theta} + b \tan \theta.
      \end{split}
    \end{equation}
    Thus $c$ and $d$ are positive combinations of $1$ and $b$, so their complex arguments lie between $0$ and the argument of $b$. Moreover, one can compute:
    \begin{equation}
      \Im\left(\frac{c}{d}\right) =  \frac{\cos^2\theta
        \Im(b)}{\left|1+b\sin \theta \right|^2} > 0
    \end{equation}
 and we deduce that the order of the arguments matches the one of Figure~\ref{fig:order}.

    \item $(i) \Rightarrow (ii)$:
Assume that $(ii)$ is not verified, that is, $\Im(b)<0$. Then consider $\bar{b},\bar{c},\bar{d}$, which is still a solution to the propagation equation as this equation has real coefficients. These values do satisfy $(ii)$, hence also $(iii)$ and thus $(i)$ by the previous points. Note that they correspond to a quadrilateral that is the image by a reflection of our desired quad $\mathcal{S}(u) \mathcal{S}(x) \mathcal{S}(v) \mathcal{S}(y)$, and therefore for our initial solution $(i)$ does not hold.
\end{itemize}
\end{proof}

\subsubsection{Star-triangle transformation on propagation equations}

This part consists in rephrasing Baxter's results on the star-triangle transformation of the Ising model (see Section 6.4 of \cite{Baxter:exactly}). It is slightly easier to present it in the triangle-star direction, \emph{i.e.} from the right-hand side to the left-hand side of Figure~\ref{fig:st}, although the converse is also possible.

Let us suppose that the graph $G$ contains a triangle, as on the right-hand side of Figure~\ref{fig:st}. We label its edges with the parameters $\theta_i$ and define $x_i, J_i$ as in \eqref{eq:defJ} and \eqref{eq:defx}. It is possible
to transform the triangle into the star displayed on the left-hand
side, while finding parameters such that both Ising models are coupled
and agree everywhere except at $A_7$. This is called the
\emph{star-triangle transformation}.

\begin{figure}[h]
  \centering
  \begin{tikzpicture}[scale=0.7]

\begin{scope}[xshift=9cm]
  \coordinate (hh) at (0 ,2) ;
  \coordinate (bd) at (1.732,-1) ;
  \coordinate (bg) at (-1.732,-1) ;
  \coordinate (hd) at ($(hh)+(bd)$) ;
  \coordinate (hg) at ($(hh)+(bg)$) ;
  \coordinate (bb) at ($(bd)+(bg)$) ;
  \node [draw=black, fill=black,thick,circle,inner sep=0pt,minimum size=4pt] at (hh) {};
  \node [draw=black, fill=white,thick,circle,inner sep=0pt,minimum size=4pt] at (hd) {};
  \node [draw=black, fill=white,thick,circle,inner sep=0pt,minimum size=4pt] at (hg) {};
  \node [draw=black, fill=black,thick,circle,inner sep=0pt,minimum size=4pt] at (bd) {};
  \node [draw=black, fill=black,thick,circle,inner sep=0pt,minimum size=4pt] at (bg) {};
  \node [draw=black, fill=white,thick,circle,inner sep=0pt,minimum size=4pt] at (bb) {};
  \node [draw=black, fill=white,thick,circle,inner sep=0pt,minimum size=4pt] at (0,0) {};
  \draw (hh) node [above] {$A_1$};
  \draw (hd) node [above right] {$A_6$};
  \draw (hg) node [above left] {$A_2$};
  \draw (bd) node [below right] {$A_5$};
  \draw (bg) node [below left] {$A_3$};
  \draw (bb) node [below] {$A_4$};
  \draw (0,0) node [below] {$A_0$};
  \draw (hh) -- (bg) -- (bd) -- (hh);
  \draw (-0.8,0.5) node [left] {$\theta_5$};
  \draw (0.8,0.5) node [right] {$\theta_3$};
  \draw (0,-1) node [below] {$\theta_1$};
\end{scope}

\begin{scope}[yshift=0cm]
  \begin{scope}[xshift=0cm]
  \coordinate (hh) at (0 ,2) ;
  \coordinate (bd) at (1.732,-1) ;
  \coordinate (bg) at (-1.732,-1) ;
  \coordinate (hd) at ($(hh)+(bd)$) ;
  \coordinate (hg) at ($(hh)+(bg)$) ;
  \coordinate (bb) at ($(bd)+(bg)$) ;
  \coordinate (cc) at ($(hd) + (bb)- (bd)$);
  \node [draw=black, fill=black,thick,circle,inner sep=0pt,minimum size=4pt] at (hh) {};
  \node [draw=black, fill=white,circle,thick,inner sep=0pt,minimum size=4pt] at (hd) {};
  \node [draw=black, fill=white,thick,circle,inner sep=0pt,minimum size=4pt] at (hg) {};
  \node [draw=black, fill=black,thick,circle,inner sep=0pt,minimum size=4pt] at (bd) {};
  \node [draw=black, fill=black,thick,circle,inner sep=0pt,minimum size=4pt] at (bg) {};
  \node [draw=black, fill=white,thick,circle,inner sep=0pt,minimum size=4pt] at (bb) {};
  \node [draw=black, fill=black,thick,circle,inner sep=0pt,minimum size=4pt] at (cc) {};
  \draw (hh) node [above] {$A_1$};
  \draw (hd) node [above right] {$A_6$};
  \draw (hg) node [above left] {$A_2$};
  \draw (bd) node [below right] {$A_5$};
  \draw (bg) node [below left] {$A_3$};
  \draw (bb) node [below] {$A_4$};
  \draw (0,0) node [below] {$A_7$};
  \draw (hh) -- (cc);
  \draw (bg) -- (cc);
  \draw (bd) -- (cc);
  \draw (0.9,-0.5) node [right] {$\theta_2$};
  \draw (-0.9,-0.5) node [left] {$\theta_6$};
  \draw (0,1) node [left] {$\theta_4$};
\end{scope}
\begin{scope}[xshift=3.5cm, yshift=-0.2cm]
  \draw[>=stealth,<->, line width = 1pt] (0,0) -- (2,0);
\end{scope}
\end{scope}

\end{tikzpicture}
  \caption{Star-triangle transformation on $G$ (black vertices)}
  \label{fig:st}
\end{figure}
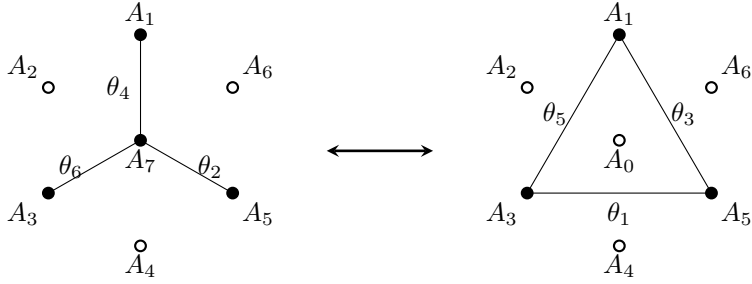

\begin{proposition}[\cite{Baxter:exactly}]
\label{prop:Baxter}
    Let $k'\in (0,\infty)$ be defined as
    \begin{equation}
      k' =
      \frac{(1-x_1^2)(1-x_3^2)(1-x_5^2)}{4\sqrt{(1+x_1x_3x_5)(x_1+x_3x_5)(x_3+x_1x_5)(x_5+x_1x_3)}}.
    \end{equation}
    Then the parameters $\theta_2,\theta_4,\theta_6$ obtained by the star-triangle transformation are the unique angles in $(0,\frac{\pi}{2})$ such that
    \begin{equation}
      \label{eq:st_theta}
      \forall i \in \{1,3,5\}, \ \tan \theta_i \tan \theta_{i+3} = \frac{1}{k'}
    \end{equation}
where the labels of the angles are considered modulo $6$.
\end{proposition}

\begin{remark}
  \label{rem:elliptic}
  This transformation may be expressed in different ways; in this
  remark, we recall a convenient elliptic parametrization due to Baxter
  \cite{Baxter:exactly}, also used in \cite{BoutillierDeTiliereRaschel}. The previous definition of $k'$ naturally comes from the use of an \emph{elliptic modulus} $k \in i\R \cup [0,1)$ such that $k^2 + k'^2 = 1$; see Appendix~\ref{sec:elliptic} for a short introduction to elliptic functions.

We can define elliptic angles $\tau_1,\dots,\tau_6$ and $\theta'_1,\dots,\theta'_6$ by
  \begin{equation}
    \begin{split}
      \tau_i &= \Fel(\theta_i,k), \\
      \theta'_i &= \frac{\pi \tau_i }{2 \Kel(k)}.
    \end{split}
  \end{equation}
where $\Fel$ (resp. $\Kel$) is the incomplete (resp. complete) integral of the first kind, see Appendix~\ref{sec:elliptic} for definitions and details. The normalization is such that both $\theta$ and $\theta'$ variables live in $(0,\frac{\pi}{2})$, while $\tau$ variables live in $(0,K(k))$. Then $k$ can be seen as the only modulus such that the $\theta'$ angles satisfy
  \begin{equation}
    \label{eq:st_sum}
    \theta'_1+\theta'_3+\theta'_5 = \frac{\pi}{2}.
  \end{equation}
In these parameters, the star-triangle transformation simply reads \cite{Baxter:exactly,BoutillierDeTiliereRaschel}
  \begin{equation}
    \label{eq:st_tau}
    \forall i \in \{1,3,5\}, \ \theta'_{i+3} = \frac{\pi}{2} - \theta'_i.
  \end{equation}
  We emphasize that this simple formula requires the introduction of a
  local elliptic modulus, while
  Equations~\eqref{eq:st_theta2},\eqref{eq:st_theta1} do not.
\end{remark}

The graphs $\dbc$ corresponding to the star and to
the triangle configurations are represented in
Figure~\ref{fig:st_propagation}. We show that when the angles $\theta$
are chosen to satisfy the star-triangle relations \eqref{eq:st_theta},
the propagation equations ``seen from the boundary vertices''
$w_1,\dots,w_6$ are the same.

\begin{figure}[h]
  \centering
  \begin{tikzpicture}[scale=0.95]
  \begin{scope}[xshift=7cm]
\draw [line width=2pt,color=gray] (-2.2992330335477806,-2.184431458745066)-- (2.3111103481729165,-2.175925290144105);
\draw [line width=2pt,color=gray] (2.3111103481729165,-2.175925290144105)-- (-0.0014279007847367098,1.8124961142949965);
\draw [line width=2pt,color=gray] (-0.0014279007847367098,1.8124961142949965)-- (-2.2992330335477806,-2.184431458745066);
\draw [dashed] (-2.6905167891919723,-1.469913296264367)-- (-1,1.5);
\draw [dashed] (-1,1.5)-- (-0.1897032205095279,0.8182460573940611);
\draw [dashed] (-0.1897032205095279,0.8182460573940611)-- (-1.5,-1.5);
\draw [dashed] (-1.5,-1.5)-- (-2.2141713475381737,-1.2827775870432316);
\draw [dashed] (-2.2141713475381737,-1.2827775870432316)-- (-0.8446782027835015,1.1159619584276856);
\draw [dashed] (-0.8446782027835015,1.1159619584276856)-- (0.2356052095385069,0.8182460573940611);
\draw [dashed] (0.2356052095385069,0.8182460573940611)-- (-1.2614804642305755,-1.8782093891104805);
\draw [dashed] (-1.2614804642305755,-1.8782093891104805)-- (-2.6905167891919723,-1.469913296264367);
\draw [dashed] (-1.5,-1.5)-- (-1.2699866328315363,-2.490653528379651);
\draw [dashed] (-1.5506901966632392,-2.7628509236103933)-- (-1.2614804642305755,-1.8782093891104805);
\draw [dashed] (-1.5506901966632392,-2.7628509236103933)-- (1.605098354293179,-2.7543447550094324);
\draw [dashed] (1.605098354293179,-2.7543447550094324)-- (1.5,-1.5);
\draw [dashed] (1.5,-1.5)-- (-1.5,-1.5);
\draw [dashed] (-1.2614804642305755,-1.8782093891104805)-- (1.3158886218605153,-1.8782093891104805);
\draw [dashed] (1.3158886218605153,-1.8782093891104805)-- (1.3669256334662794,-2.490653528379651);
\draw [dashed] (1.3669256334662794,-2.490653528379651)-- (-1.2699866328315363,-2.490653528379651);
\draw [dashed] (2.2515671679661917,-1.3168022614470745)-- (1.3158886218605153,-1.8782093891104805);
\draw [dashed] (1.5,-1.5)-- (2.7364187782209513,-1.4614071276634064);
\draw [dashed] (2.7364187782209513,-1.4614071276634064)-- (1.1117405754374585,1.473221039668035);
\draw [dashed] (1.1117405754374585,1.473221039668035)-- (-0.1897032205095279,0.8182460573940611);
\draw [dashed] (0.2356052095385069,0.8182460573940611)-- (0.9416172034182446,1.0989496212257641);
\draw [dashed] (0.9416172034182446,1.0989496212257641)-- (2.2515671679661917,-1.3168022614470745);
\draw [dashed] (1.3158886218605153,-1.8782093891104805)-- (-0.1897032205095279,0.8182460573940611);
\draw [dashed] (1.5,-1.5)-- (0.2356052095385069,0.8182460573940611);

\draw [color=gray] (-1.1,-0.17) node [left] {$\theta_5$};
\draw [color=gray] (0,-2.1) node [below] {$\theta_1$};
\draw [color=gray] (1.1,-0.17) node [right] {$\theta_3$};

\draw (-1,1.5) node [] {$w_1$};
\draw (-2.7,-1.5) node [] {$w_2$};
\draw (-1.55,-2.8) node [] {$w_3$};
\draw (1.55,-2.8) node [] {$w_4$};
\draw (2.7,-1.5) node [] {$w_5$};
\draw (1.1,1.5) node [] {$w_6$};

\draw (-0.2,0.8) node [] {$y_1$};
\draw (-1.2,-1.9) node [] {$y_3$};
\draw (1.5,-1.5) node [] {$y_5$};
\end{scope}

\begin{scope}[]
\draw [line width=2pt,color=gray] (-2.2992330335477806,-2.184431458745066)-- (0.0034831379467993783,-0.849286878198058);
\draw [line width=2pt,color=gray] (0.0034831379467993783,-0.849286878198058)-- (-0.0014279007847367098,1.8124961142949965);
\draw [line width=2pt,color=gray] (0.0034831379467993783,-0.849286878198058)-- (2.3111103481729165,-2.175925290144105);
\draw [dashed] (-0.9467522259950308,-0.3981360525433185)-- (-0.8787028771873452,1.6773690860910926);
\draw [dashed] (-0.8787028771873452,1.6773690860910926)-- (0.8395431802067154,1.6943814232930139);
\draw [dashed] (0.8395431802067154,1.6943814232930139)-- (0.8140246744038332,-0.3896298839423578);
\draw [dashed] (0.8140246744038332,-0.3896298839423578)-- (-0.4959252901441139,-0.610790267567336);
\draw [dashed] (-0.4959252901441139,-0.610790267567336)-- (-0.46190061574027114,1.4051716908603502);
\draw [dashed] (-0.46190061574027114,1.4051716908603502)-- (0.4482594245625233,1.4051716908603502);
\draw [dashed] (0.4482594245625233,1.4051716908603502)-- (0.40572858155771985,-0.610790267567336);
\draw [dashed] (0.40572858155771985,-0.610790267567336)-- (-0.9467522259950308,-0.3981360525433185);
\draw [dashed] (-2.392800888158349,-1.3933577788557205)-- (-1.6272457140718866,-2.558702877187337);
\draw [dashed] (-1.6272457140718866,-2.558702877187337)-- (0.03145716311544922,-1.6145181624806988);
\draw [dashed] (0.03145716311544922,-1.6145181624806988)-- (-0.4959252901441139,-0.610790267567336);
\draw [dashed] (-0.4959252901441139,-0.610790267567336)-- (-1.8824307721007074,-1.5464688136730131);
\draw [dashed] (-1.8824307721007074,-1.5464688136730131)-- (-1.5166655222593974,-2.1333944471393016);
\draw [dashed] (-1.5166655222593974,-2.1333944471393016)-- (0.005938657312567138,-1.2147282382355458);
\draw [dashed] (0.005938657312567138,-1.2147282382355458)-- (-0.9467522259950308,-0.3981360525433185);
\draw [dashed] (-0.9467522259950308,-0.3981360525433185)-- (-2.392800888158349,-1.3933577788557205);
\draw [dashed] (0.03145716311544922,-1.6145181624806988)-- (1.7326908833075885,-2.6097398887931007);
\draw [dashed] (1.7326908833075885,-2.6097398887931007)-- (2.557789237600776,-1.4699132962643668);
\draw [dashed] (2.557789237600776,-1.4699132962643668)-- (0.8140246744038332,-0.3896298839423578);
\draw [dashed] (0.8140246744038332,-0.3896298839423578)-- (0.005938657312567138,-1.2147282382355458);
\draw [dashed] (0.005938657312567138,-1.2147282382355458)-- (1.690160040302785,-2.175925290144105);
\draw [dashed] (1.690160040302785,-2.175925290144105)-- (2.0474191215431343,-1.580493488076856);
\draw [dashed] (2.0474191215431343,-1.580493488076856)-- (0.40572858155771985,-0.610790267567336);
\draw [dashed] (0.40572858155771985,-0.610790267567336)--
(0.03145716311544922,-1.6145181624806988);

\draw [color=gray] (0.1,0.5) node [left] {$\theta_4$};
\draw [color=gray] (-1,-1.5) node [above] {$\theta_6$};
\draw [color=gray] (1,-1.5) node [above] {$\theta_2$};

\draw (-0.9,1.7) node [] {$w_1$};
\draw (-2.4,-1.4) node [] {$w_2$};
\draw (-1.6,-2.6) node [] {$w_3$};
\draw (1.7,-2.6) node [] {$w_4$};
\draw (2.6,-1.5) node [] {$w_5$};
\draw (0.8,1.7) node [] {$w_6$};

\draw (-1.1,-0.4) node [] {$z_2$};
\draw (0,-1.8) node [] {$z_4$};
\draw (1,-0.4) node [] {$z_6$};

\end{scope}
\end{tikzpicture}

  \caption{Star-triangle transformation of $\dbc$.}
  \label{fig:st_propagation}
\end{figure}
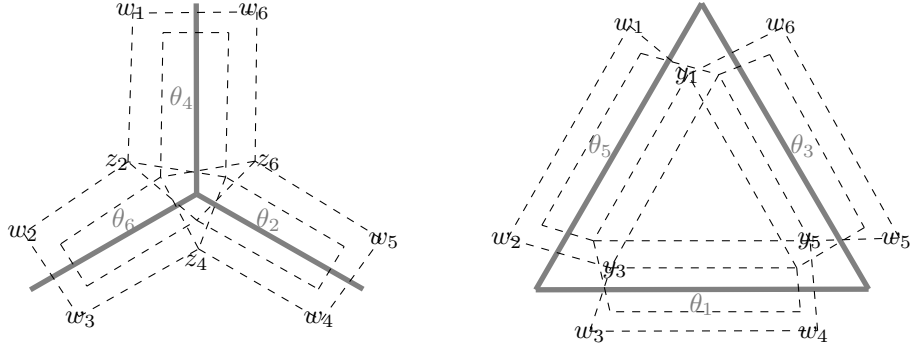

\begin{proposition}
  \label{prop:st_prop}
  Suppose that $\theta_1,\dots,\theta_6 \in (0,\frac{\pi}{2})$ satisfy the star-triangle relations \eqref{eq:st_theta}. Let $(w_1,w_2,w_3,w_4,w_5,w_6,y_1,y_3,y_5)$ be a solution of the propagation equation on the ``triangle'' graph on the right of Figure~\ref{fig:st_propagation}. Then there exists a unique triple $(z_2,z_4,z_6)$ such that $(w_1,w_2,w_3,w_4,w_5,w_6,z_2,z_4,z_6)$ is
  a solution of the propagation equation on the ``star'' graph on the left. Conversely, for any solution
  $(w_1,w_2,w_3,w_4,w_5,w_6,z_2,z_4,z_6)$ on the left, there exists a unique triple $y_1,y_3,y_5$ such that
  $(w_1,w_2,w_3,w_4,w_5,w_6,y_1,y_3,y_5)$ is a solution on the right.
\end{proposition}

\begin{proof}
It is easy to see that any values of $(y_1,y_3,y_5)$ uniquely characterize the solution of the propagation equation on the triangle graph. Thus, the set of possible vectors $(w_1,\dots, w_6)$ is a $3$-dimensional subspace $V$. By setting $(y_1,y_3,y_5)$ to be the elements of the canonical basis of $\C^3$ and solving the propagation equation, we get a basis of $V$:
  \begin{equation}
    u_1=\begin{pmatrix} 1/\cos \theta_5 \\ \tan \theta_5 \\ 0 \\
      0 \\ \tan\theta_3 \\ 1/\cos\theta_3
    \end{pmatrix},
    u_3=\begin{pmatrix} \tan \theta_5 \\ 1/\cos \theta_5 \\ 1/\cos\theta_1 \\
      \tan\theta_1 \\ 0 \\ 0
    \end{pmatrix},
    u_5=\begin{pmatrix} 0 \\ 0 \\ \tan \theta_1 \\
      1/\cos \theta_1 \\ 1/\cos\theta_3 \\ \tan\theta_3
    \end{pmatrix}.
  \end{equation}
  Similarly, the set of values of $(w_1,\dots,w_6)$ for the star graph
  is a subspace $V'$ with basis
  \begin{equation}
    v_2=\begin{pmatrix} 1/\sin \theta_4 \\ 1/\sin \theta_6 \\ 1/\tan \theta_6 \\
      0 \\ 0 \\ 1/\tan \theta_4
    \end{pmatrix},
    v_4=\begin{pmatrix} 0 \\ 1/\tan \theta_6 \\ 1/\sin \theta_6 \\ 1/\sin\theta_2 \\
      1/\tan\theta_2 \\ 0
    \end{pmatrix},
    v_6=\begin{pmatrix} 1/\tan \theta_4 \\ 0 \\ 0 \\
      1/\tan \theta_2 \\ 1/\sin\theta_2 \\ 1/\sin\theta_4
    \end{pmatrix}.
  \end{equation}

We want to show that these subspaces are equal. By an argument of dimension, we just need to show that $v_2,v_4,v_6 \in V$. Let us do it for $v_2$, the other cases being symmetric. We claim that
  \begin{equation}
    \frac{1}{\cos \theta_3 \tan \theta_4} u_1 + \frac{1}{\cos \theta_1
      \tan \theta_6} u_3 - \frac{\tan \theta_1}{\tan \theta_6} u_5  = v_2.
  \end{equation}
This is checked immediately for the third and fourth entries; the fifth and sixth entries are similar after noting that by \eqref{eq:st_theta},
  \begin{equation}
    \label{eq:tan_ratio}
    \frac{\tan \theta_1}{\tan \theta_6} = \frac{\tan \theta_3}{\tan \theta_4}.
  \end{equation}
The first and second entries are a bit more tedious.
  For the first one, we want to show that
  \begin{equation}
    \label{eq:entry_1}
    \frac{1}{\cos \theta_3 \tan \theta_4 \cos \theta_5} + \frac{\tan
      \theta_5}{\cos \theta_1
      \tan \theta_6} = \frac{1}{\sin \theta_4}.
  \end{equation}
  In terms of the elliptic  variables ${\tau}_i = F(\theta_i,k)$ (see
  Appendix \ref{sec:elliptic}), this amounts to showing (we omit the
  elliptic parameter $k$):
  \begin{equation}
    \label{eq:ell_c}
    \nc({\tau}_3)\cs({\tau}_4)\nc({\tau}_5)
    + \nc({\tau}_1) \sc({\tau}_5)
    \cs({\tau}_6) - \ns({\tau}_4) = 0.
  \end{equation}
  By \eqref{eq:st_tau}, ${\tau}_4 = \Kel(k) - {\tau}_1$
  and ${\tau}_6 = \Kel(k) - {\tau}_3$, and by
  \eqref{eq:st_sum}, ${\tau}_1 = \Kel(k) - ({\tau}_3 +
  {\tau}_5)$. Thus we can express all the arguments in
  terms of ${\tau}_3,{\tau}_5$. Using the change
  of arguments in elliptic functions \eqref{eq:ell_change_var}, the left-hand side of \eqref{eq:ell_c} is
  equal to
  \begin{equation}
    \begin{split}
      & \nc({\tau}_3)\nc({\tau}_5)\cs({\tau}_3+{\tau}_5)
      + \sc({\tau}_3) \sc({\tau}_5)
      \ds({\tau}_3 + {\tau}_5)
      - \ns({\tau}_3+{\tau}_5) \\
      =
      &\nc({\tau}_3)\nc({\tau}_5)\ns({\tau}_3+{\tau}_5)
      \times \\
      & \ \ \ \left( \cn({\tau}_3+{\tau}_5) +
        \sn({\tau}_3)\sn({\tau}_5)
        \dn({\tau}_3+{\tau}_5) -
        \cn({\tau}_3)\cn({\tau}_5) \right).
    \end{split}
  \end{equation}
  By \eqref{eq:ell_add}, this is equal to zero.

  For the second entry, we want to show that
  \begin{equation}
    \label{eq:entry_2}
     \frac{\tan
      \theta_5}{\cos \theta_3
      \tan \theta_4} + \frac{1}{\cos \theta_1 \tan \theta_6 \cos \theta_5}= \frac{1}{\sin \theta_6}.
  \end{equation}
  By \eqref{eq:st_theta}, $\frac{\tan \theta_5}{\tan \theta_4} =
  \frac{\tan \theta_1}{\tan \theta_2}$. Hence \eqref{eq:entry_2} is
  equivalent to
  \begin{equation}
     \frac{\tan
      \theta_1}{\cos \theta_3
      \tan \theta_2} + \frac{1}{\cos \theta_1 \tan \theta_6 \cos \theta_5}= \frac{1}{\sin \theta_6}.
  \end{equation}
  Up to a cyclic shift, this is the same as \eqref{eq:entry_1}, so it
  holds by the previous discussion.
\end{proof}

As a byproduct of the previous proof, we obtain Theorem~\ref{thm:newformula}, which provides formulas expressing the change of $\theta$ parameters in the star-triangle transformation. These formulas are much simpler than the classical computation described in Proposition \ref{prop:Baxter} and to the best of our knowledge, they are new.

\begin{proof}[Proof of Theorem~\ref{thm:newformula}]
For $i=4$, this follows from combining equations \eqref{eq:tan_ratio} and \eqref{eq:entry_1}. Cyclic shifts of indices give the cases $i=2$ and $i=6$.

To obtain the other three values of $i$, we apply the Kramers-Wannier duality of the Ising model \cite{KramersWannier}, which has the effect of transforming the variables $\theta_i$ into $\frac{\pi}{2}-\theta_i$. In this duality, the star graph becomes its dual, \emph{i.e.} a triangle, and vice-versa. Hence the formula can be deduced from the previous one by changing sines into cosines and vice-versa.
\end{proof}

\medskip

We now have all the elements to prove the unique proper flip property of $1$-embeddings.

\begin{proof}[Proof of Theorem~\ref{theo:1quad}]
  We start with $A_0,A_1,\dots,A_6$, a proper embedding of the left-hand side of Figure~\ref{fig:YDelta}. As uniqueness is a consequence of Proposition~\ref{prop:uniqueness}, we just have to prove that there exists a point $A_7$ such that $A_1,A_2,\dots, A_7$ is a proper embedding of the right-hand side.

By Proposition~\ref{prop:chelkak}, there exists a solution
$(w_1,w_2,w_3,w_4,w_5,w_6,y_1,y_3,y_5)$ of the propagation equation as
on the left-hand side of Figure~\ref{fig:st_propagation} such that the
points $A_0,A_1,\dots,A_6$ are the $s$-embedding of this
solution. Hence, by Proposition~\ref{prop:st_prop}, there exists
$(z_2,z_4,z_6)$ such that $(w_1,w_2,w_3,w_4,w_5,w_6,z_2,z_4,z_6)$ is a
solution to the propagation equation of the right-hand side of
Figure~\ref{fig:st_propagation}. Let us consider its $s$-embedding. It
has the same boundary as the initial one --- hence the points
$A_1,\dots,A_6$ are unchanged --- and we have a new point $A_7$. It
remains to prove that the three new $1$-quads are proper and oriented
as on the right-hand side of Figure~\ref{fig:YDelta}. Then
  $w_1/w_2$ is not a real number. Indeed, if it were the case, the
  propagation equations would imply that the variables
  $(w_1,w_2,w_3,w_4,w_5,w_6,z_2,z_4,z_6)$ all have the same
  argument, and the initial embedding would not be proper. Moreover the initial quad $A_1A_2A_3A_0$ is proper and positively oriented. Hence we can apply Lemma~\ref{lemma:orient} to obtain the ordering of the arguments of $w_1,w_2,y_3,y_1$. By doing the same for the other two initial quads, the order of the arguments of the variables $w_i$ and $y_i$ is that on the left of Figure~\ref{fig:order_x}.

  \begin{figure}[h]
    \centering
    \def\svgwidth{12cm}
    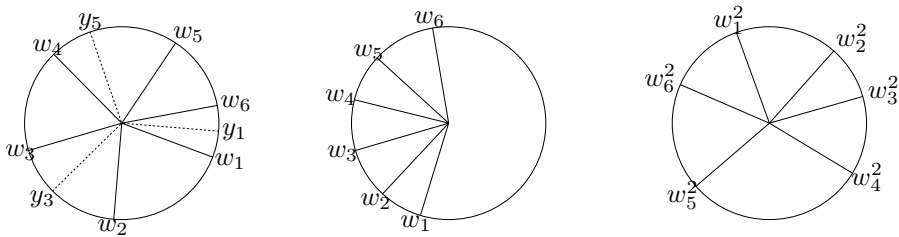
    \caption{The cyclic order of the arguments of the variables $w_i$ and $y_i$.}
    \label{fig:order_x}
  \end{figure}

  Let us prove that the new (proper) quad $A_1A_2A_7A_6$ is positively oriented. By Lemma~\ref{lemma:orient}, it is enough to prove that $\Im(w_6/w_1)>0$. If this is not the case, then we are in the situation of the second configuration of Figure~\ref{fig:order_x}, where all the arguments are included in a half-circle. Therefore we can get the order of the arguments of the $w_i^2$, as in the picture on the right of Figure~\ref{fig:order_x}. However, as in the proof of Lemma~\ref{lemma:orient}, the successive internal angles of the hexagon $A_1A_2A_3A_4A_5A_6$ can be expressed as the successive oriented angles in direct order in this last figure, and in particular their sum is $2\pi$. This contradicts the fact that the hexagon is non-crossed, as the sum should be $4\pi$. By symmetry, the three new quads are properly oriented, which concludes the proof.
\end{proof}

\begin{remark}
The so-called \emph{exact bosonization} operation provides a map from
two independent Ising models on some planar graph $G$ to a bipartite dimer model on a modified graph $\tilde{G}$ \cite{Dubedat}. Under this correspondence, the Ising star-triangle move corresponds to the composition of six local moves for the dimer model called urban renewals (\cite{KenyonPemantle}, Fig. 6). On the other hand, embeddings as centers of circle patterns are known to be adapted to bipartite dimer models and the counterpart of urban renewal for these geometric objects is called the Miquel move \cite{KLRR}. From all this, we conclude that it is possible to write the star-triangle move for $s$-embeddings as a composition of six Miquel moves.
\end{remark}

\section{The cube move for $\alpha>1$}
\label{sec:>1}

In this section, we prove that the set of $\alpha$-realizations for
$\alpha>1$ satisfies the flip property. In order to achieve this, we
first study basic properties of construction curves, namely (un)boundedness, connectedness and asymptotic behaviour.

\subsection{Properties of construction curves}

Recall that construction curves for
$\alpha \in \{-\infty,0,1,2,+\infty\}$ are completely characterized by
Proposition~\ref{prop:special_construction}.  In the rest of this
section, we restrict ourselves to $\alpha \in \R\setminus\{0,1,2\}$.

Without loss of generality, to study an $\alpha$-construction curve with foci $A,B$, we may assume that $A = (0,-1)$ and $B=(0,1)$. Then any $\alpha$-construction curve is characterized by a parameter $\lambda \in \R$, such that the
curve is the set of points
$\left\{ M \in \R^2 \mid MA^\alpha - MB^\alpha = \lambda
\right\}$. Hence, it is a level set of the function
$F_\alpha: M\mapsto MA^\alpha - MB^\alpha$, or
\begin{align*}
F_{\alpha}(x,y)=\left( x^2+(y+1)^2 \right)^{\alpha/2} - \left( x^2+(y-1)^2 \right)^{\alpha/2}.
\end{align*}

We denote this curve by $C(\alpha,\lambda):=F_\alpha^{-1}(\{\lambda\})$. Remark that when $\lambda=0$ the curve $C(\alpha,0)$ is $P(AB)$, which in that case is the horizontal axis. Moreover, changing $\lambda$ into $-\lambda$ has the effect of reflecting the curve across $P(AB)$.
Finally, note that the curve may be empty for some values of $\alpha$ and $\lambda$.

The following result is immediate:
\begin{lemma}
\label{lemma:sym}
Let $M$ be a point distinct from $A,B$. Let $M'$ be the image of $M$ by the reflection across $(AB)$, and $M''$ the image of $M$ by the reflection across $P(AB)$. Then one has $F_{\alpha}(M) = F_{\alpha}(M') = -F_{\alpha}(M'')$. In particular, $F_{\alpha}$ is the zero function on $P(AB)$. Moreover $F_{\alpha}$ has a constant sign on each half-plane bounded by $P(AB)$.
\end{lemma}

\begin{corollary}
For $\alpha \in\R^*$ and $\lambda > 0$, $C(\alpha,\lambda)$ is symmetric with respect to the axis $(AB)$ and remains on one side of $P(AB)$ (namely, the side containing $B$ if $\alpha>0$ and the side containing $A$ if $\alpha<0$).
\end{corollary}

In order to understand the level sets of the function $F_\alpha$, let us study its profile. As a result of the symmetries of Lemma~\ref{lemma:sym}, it is sufficient to do so in the quadrant $\left(\R_+\right)^2$.

\begin{lemma}
\label{lemma:study_F}
For any $M = (x_M,y_M) \in \left(\R_+\right)^2$,
\begin{itemize}
\item[(i)] if $\alpha<0$:
\begin{itemize}
\item $F_\alpha$ is $C^{\infty}$ on $\R^2 \setminus \{A,B\}$.

\item $\lim_{M \to B} F_{\alpha}(M) = -\infty$ and $\lim_{x_M^2+y_M^2\to \infty} F_{\alpha}(M) = 0$.

\item if $y_M>0$ and $M\neq B$, then $F_{\alpha}(M) < 0$.

\item if $x_M>0$ and $y_M>0$, $\frac{\partial F_\alpha}{\partial x}(M) > 0$.

\item if $x_M=0$: if $y_M<1$, $\frac{\partial F_\alpha}{\partial y}(M) < 0$; if $y_M>1$, $\frac{\partial F_\alpha}{\partial y}(M) > 0$.
\end{itemize}

\item[(ii)] if $0<\alpha<1$:

\begin{itemize}
\item $F_\alpha$ is continuous on $\R^2$, and $C^{\infty}$ on $\R^2 \setminus \{A,B\}$.

\item $\lim_{x_M^2+y_M^2\to \infty} F_{\alpha}(M) = 0$.

\item if $y_M>0$, $F_{\alpha}(M) > 0$.

\item if $x_M>0$ and $y_M>0$, $\frac{\partial F_\alpha}{\partial x}(M) < 0$.

\item if $x_M=0$: if $y_M<1$, $\frac{\partial F_\alpha}{\partial y}(M) > 0$; if $y_M>1$, $\frac{\partial F_\alpha}{\partial y}(M) < 0$.

\end{itemize}

\item[(iii)] if $\alpha>1$, with $\alpha \neq 2$:

\begin{itemize}

\item $F_\alpha$ is $C^1$ on $\R^2$, and $C^{\infty}$ on $\R^2 \setminus \{A,B\}$.

\item for any fixed $x_M \geq 0$, $\lim_{y_M\to \infty} F_{\alpha}(M) = +\infty$.

\item for $y_M>0$, $F_{\alpha}(M) > 0$.

\item if $x_M>0$ and $y_M>0$, for $\alpha<2$, $\frac{\partial F_\alpha}{\partial x}(M) < 0$ and for $\alpha>2$, $\frac{\partial F_\alpha}{\partial x}(M) > 0$.

\item if $x_M,y_M \geq 0$, $\frac{\partial F_\alpha}{\partial y}(M) > 0$.
\end{itemize}
\end{itemize}
\end{lemma}

\begin{proof}
Firstly, all the claims on the regularity of $F_\alpha$ are clear. Furthermore, the sign of this function on the quadrant $(\R_+)^2$ is apparent as well, since $MA \geq MB$ with equality only if $x_M=0$.

Let us now deal with the asymptotic behaviour of $F_\alpha$. For $\alpha<0$, $MB^\alpha$ goes to $+\infty$ as $M \rightarrow B$ and therefore $\lim F_\alpha = -\infty$. On the other hand, for any $\alpha<1$, notice that by the mean value theorem, for any $M$ there exists a $u \in [MA,MB]$ such that
\[ F_{\alpha}(M) = \left(MA-MB\right) \alpha u^{\alpha-1} \leq \alpha AB . MB^{\alpha - 1},\]
where we used the triangular inequality. This goes to $0$ as $MB$ grows. In the case $\alpha>1$, let $x_M$ be fixed. For any large enough $y_M$, by the mean value theorem, there exists a $u\in [y_M-1,y_M+1]$ such that
\[F_{\alpha}(M) = u \alpha \left(x^2+u^2\right)^{\alpha/2 - 1}.\]
For large $y_M$, this is equivalent to $\alpha y_M^{\alpha-1}$ and so tends to $+\infty$.

We finally consider the partial derivatives of $F_\alpha$. For any $M \in (\R_+^*)^2$, one can compute:
\[\frac{\partial F_\alpha}{\partial x}(M) = \alpha x_M \left(MA^{\alpha-2} - MB^{\alpha-2}\right)\]
and its sign is apparent.
Regarding the second partial derivative of $F_\alpha$, we compute:
\begin{equation}
\label{eq:partial_y}
\frac{\partial F_\alpha}{\partial y} (M) = \alpha (y_M+1) \left( x_M^2+(y_M+1)^2 \right)^{\frac{\alpha}{2}-1} - \alpha (y_M-1) \left( x_M^2+(y_M-1)^2 \right)^{\frac{\alpha}{2}-1}.
\end{equation}
Let us investigate it in the case $\alpha<1$. When $x_M = 0$ and $y_M \neq 1$, this reduces to
\[\alpha(\sgn(y_M+1) |y_M+1|^{\alpha-1} - \sgn(y_M-1) |y_M-1|^{\alpha-1}).\]
A study of the function $t \mapsto \sgn(t)|t|^{\alpha-1}$ shows that, for $\alpha \in (0,1)$, this quantity is positive when $y_M<1$ and negative when $y_M>1$ (and conversely for $\alpha<0$).

Finally, we treat the case $\alpha>1$. Using the expression \eqref{eq:partial_y}, in order to prove that $\frac{\partial F_\alpha}{\partial y} (M)>0$, we only have to prove that $g : u \rightarrow u \left( x_M^2+u^2 \right)^{\alpha/2-1}$ is increasing for any $x_M\geq 0$. For $x_M=0$ the previous argument applies. For $x_M>0$, $g$ is differentiable and, for $u \in \R$, $g'(u) = [x_M^2+(\alpha-1)u^2] \left( x_M^2+u^2 \right)^{\alpha/2-2} >0$. Hence, $g$ is increasing and the result follows.
\end{proof}

The study of the function $F_\alpha$ allows one to prove that the curve $C(\alpha,\lambda)$ is obtained from the graph of a real function:

\begin{corollary}
  \label{cor:1}
Let $\alpha \in \R\setminus\{0,1,2\}$ and $\lambda \in \R^*$. We assume that $\alpha \lambda > 0$, which ensures that the curve $C(\alpha,\lambda)$ is in the upper half-plane.

For $\alpha<1$, the curve $C(\alpha,\lambda)$ is bounded. If the curve is nonempty, there exists $0<y_-<1<y_+$ such that $C(\alpha,\lambda) \cap (AB) = \left\{ (0,y_-), (0,y_+) \right\}$. For any $y \in [y_-,y_+]$, there exists a unique $g_\alpha(y) \geq 0$ such that $(g_\alpha(y),y)\in C(\alpha,\lambda)$. Furthermore, the function
$g_\alpha$ is $C^1$ on $(y_-,y_+)$.

For $\alpha>1$, the curve $C(\alpha,\lambda)$ is unbounded and, for any $x \in \R$, there exists a unique $f_\alpha(x) \in \R$ such that $(x,f_\alpha(x))\in C(\alpha,\lambda)$. The function $f_\alpha$ is
$C^1$ on $\R$. Moreover, for all $x\in\R, f_\alpha(x)>0$, and on the interval $[0,\infty)$, the function $f_\alpha$ is increasing for $1<\alpha<2$, decreasing for $\alpha>2$.
\end{corollary}

\begin{proof}
The (un)boundedness of the curve follows from the limits of Lemma~\ref{lemma:study_F}, while the existence of $g_\alpha$, $f_\alpha$ follows from the intermediate value theorem. In addition, by the implicit function theorem, all these functions are $C^1$. From the definition of the construction curve, we get that, for any $x \in \R$:
\begin{align*}
\frac{\partial F_\alpha}{\partial x} (x,f_\alpha(x)) + f_\alpha'(x) \frac{\partial F_\alpha}{\partial y} (x,f_\alpha(x)) = 0.
\end{align*}
By Lemma~\ref{lemma:study_F}, for all $x \in \R$, $\frac{\partial F_\alpha}{\partial x} (x,f_\alpha(x)) > 0$; moreover $\frac{\partial F_\alpha}{\partial y} (x,f_\alpha(x)) >0$ if $\alpha > 2$ and $\frac{\partial F_\alpha}{\partial y} (x,f_\alpha(x)) <0$ if $1 < \alpha < 2$. The result follows.
\end{proof}

\begin{corollary}
  \label{cor:connected}
$C(\alpha, \lambda)$ is a connected curve.
\end{corollary}

\begin{proof}
By Lemma~\ref{lemma:study_F}, for any $(x,y) \in \R^2$, $\frac{\partial F}{\partial y} (x,y)>0$. We can therefore use the implicit function theorem along with Corollary~\ref{cor:1} to get the result.
\end{proof}

We now investigate the asymptotic behaviour of the function $f_\alpha$ at $+\infty$ (and thus the asymptotic direction of the associated construction curve). The following states that, in the case $1<\alpha<2$, the construction curve admits the perpendicular bisector of $[AB]$ as asymptotic direction, and in the case $\alpha>2$, it is the actual asymptote of the curve.

\begin{lemma}
\label{lemma:asymp}
Let $\alpha\in(1,\infty)\setminus\{2\}$, and $\lambda \neq 0$. As $|x| \rightarrow \infty$, $f_\alpha(x)$ has the following asymptotic behaviour :
\begin{align*}
f_\alpha(x) \sim \frac{\lambda}{2\alpha} |x|^{2-\alpha}.
\end{align*}
\end{lemma}

\begin{proof}
By the symmetry properties of $C(\alpha,\lambda)$, we only have to focus on the case $x \rightarrow +\infty$. For any $x>0$, by the mean value theorem, there exists an $u_x \in (f_\alpha(x)-1,f_\alpha(x)+1)$ such that
\begin{align*}
\lambda &= \left( x^2+(f_{\alpha}(x)+1)^2 \right)^{\alpha/2} - \left( x^2+(f_{\alpha}(x)-1)^2 \right)^{\alpha/2} \\
 &= 2 \alpha u_x (x^2+u_x^2)^{\alpha/2-1}\\
&= 2 \alpha u_x^{\alpha-1} \left( \frac{u_x}{x} \right)^{2-\alpha} \left(1+\left( \frac{u_x}{x} \right)^2 \right)^{\alpha/2-1}
\end{align*}

For $1<\alpha<2$, this implies that $\frac{u_x}{x}$ goes to $0$ as $x \rightarrow \infty$; otherwise the right-hand term would take arbitrarily large values. From this observation, we get that $u_x^{\alpha-1} \left( \frac{u_x}{x} \right)^{2-\alpha} = \lambda (1+o(1))$ and therefore that $u_x \sim \frac{\lambda}{2\alpha} x^{2-\alpha}$.
Since $u_x \in (f_{\alpha}(x)-1,f_{\alpha}(x)+1)$, the same holds for $f_{\alpha}(x)$.

For $\alpha>2$, as $x \rightarrow \infty$, $u_x$ has to converge to $0$, and therefore $f_\alpha(x)$ is bounded. Asymptotically, as $x \rightarrow \infty$,
\begin{align*}
\lambda &= \left( x^2+(f_{\alpha}(x)+1)^2 \right)^{\alpha/2} - \left( x^2+(f_{\alpha}(x)-1)^2 \right)^{\alpha/2} \\
& \sim \frac{\alpha}{2} x^{\alpha-2} \left[ (f_{\alpha}(x)+1)^2-(f_{\alpha}(x)-1)^2 \right] \\
& \sim 2\alpha f_{\alpha}(x) x^{\alpha-2},
\end{align*}
and the result follows.
\end{proof}

\subsection{Flip property for $\alpha > 1$}

We now go back to the study of the flip property of $\alpha$-realizations. We get the following result in the case $\alpha>1$:

\begin{theorem}
  \label{theo:>1}
For any real number $\alpha > 1$, the set of $\alpha$-realizations satisfies the flip property.
\end{theorem}

\begin{proof}
Fix $\alpha>1$ and consider six distinct points $A_1,\dots,A_6$ such
that $A_1,A_3,A_5$ (resp. $A_2,A_4,A_6$) are not aligned. Suppose that there is a point $A_0$ such that the quadrilaterals $A_1A_2A_3A_0$, $A_3A_4A_5A_0$ and $A_5A_6A_1A_0$ are all $\alpha$-quads. By Remark~\ref{rem:2curves} it suffices to show that the $\alpha$-construction curves of the hexagon $\calC_1$ and $\calC_3$ have a common point. By Lemma~\ref{lemma:asymp}, $\calC_1$ admits $P(A_2A_6)$ as an asymptotic direction and $\calC_3$ admits $P(A_2A_4)$ as an asymptotic direction. These two lines are not parallel, since $A_2,A_4,A_6$ are not aligned. Therefore, considering the behaviour of $\calC_1$ and $\calC_3$ at infinity and using Corollary~\ref{cor:connected}, these two construction curves have to intersect at least once.
\end{proof}

\begin{remark}
It may be interesting to bound the number of intersection points of two construction curves having a focus in common, in order to control the number of potential points that could appear. Based on simulations, we expect that for $\alpha<1$ there are at most two points of intersection, and for $\alpha>1$ there are at most three (see the examples of Figure~\ref{fig:ex_intersections}). It may be possible to use an analysis similar to that of Lemma~\ref{lemma:study_F} to get properties of convexity of the construction curves and deduce the previous bounds.
\end{remark}

\section{The space of $\alpha$-quads and $f$-quadrilaterals}
\label{sec:space}

In this last section, we provide different geometric characterizations of $\alpha$-quads. It appears that one of them holds in the broader context of $f$-quadrilaterals (see Definition \ref{def:fquad}). The first characterization of $\alpha$-quads is related to the existence in the quadrilateral of a so-called \textit{extremal pair}.

\begin{definition}
A pair of opposite sides of a quadrilateral is called an \emph{extremal pair} if these two sides achieve both the maximum and the minimum of the side lengths of the quadrilateral.
\end{definition}

\begin{proposition}
\label{prop:union}
A quadrilateral is an $\alpha$-quad for some $\alpha\in\overline{\R}$ if and only if it has an extremal pair.
\end{proposition}

\begin{proof}
Let $Q$ be a quadrilateral whose side lengths are denoted by $\ell_1,\ell_2,\ell_3,\ell_4$ in cyclic order (starting from any of them). Assume that $Q$ has no extremal pair. Then, without loss of generality, one may assume that $\ell_1$ is the maximal length and $\ell_4$ the minimal length, and they are distinct. Since $(\ell_1,\ell_3)$ is not an extremal pair, $\ell_3$ does not achieve the minimum. Thus, $\ell_3>\ell_4$ and $Q$ is not a $-\infty$-quad. Similarly, $\ell_2$ does not achieve the maximum; thus, $\ell_2<\ell_1$ and $Q$ is not a $+\infty$-quad. By these inequalities, $\ell_1^\alpha+\ell_3^\alpha>\ell_2^\alpha+\ell_4^\alpha$ if $\alpha>0$ and $\ell_1^\alpha+\ell_3^\alpha<\ell_2^\alpha+\ell_4^\alpha$ if $\alpha<0$ (and in any of these cases $Q$ is not an $\alpha$-quad). Furthermore, $\ell_1 \ell_3>\ell_2 \ell_4$ and $Q$ is not a $0$-quad. Hence, $Q$ is not an $\alpha$-quad for any $\alpha\in\overline{\R}$.

Conversely, suppose that $Q$ has an extremal pair. We can assume that $\ell_1$ is the maximal length, $\ell_3$ the minimal length, and that $\ell_2 \leq \ell_4$ (up to a possible mirror symmetry). If $\ell_2=\ell_3$ or if $\ell_4=\ell_1$ then $Q$ is respectively a $-\infty$-quad or a $+\infty$-quad. Hence, we can assume that $\ell_3 < \ell_2 \leq \ell_4 < \ell_1$. Consider now the function $g$ on $\R^*$ defined as
\begin{equation}
g: \alpha \mapsto \frac{\ell_1^{\alpha} + \ell_3^{\alpha} - \ell_2^{\alpha} - \ell_4^{\alpha}}{\alpha}.
\end{equation}
The function $g$ can be extended to a continuous function on $\R$ by setting $g(0)=\ln\left(\frac{\ell_1 \ell_3}{\ell_2 \ell_4}\right)$. By the previous inequalities, for $\alpha\to -\infty$, $g(\alpha) \sim \frac{\ell_3^{\alpha}}{\alpha} <0$ and for $\alpha\to +\infty$, $g(\alpha) \sim \frac{\ell_1^{\alpha}}{\alpha} >0$. Hence, by the intermediate value theorem, there exists an $\alpha\in\R$ such that $g(\alpha)=0$ and $Q$ is an $\alpha$-quad.
\end{proof}

We now characterize quadrilaterals that are $\alpha$-quads for at least two distinct values of $\alpha$.

\begin{definition}
A quadrilateral $ABCD$ is said to be a \emph{kite} if its side lengths satisfy $$\{AB,CD\}=\{BC,DA\}.$$
\end{definition}

Remark in particular that a kite is an $\alpha$-quad for all values of $\alpha \in \overline{R}$.

\begin{proposition}
\label{prop:intersection}
Let $\alpha \neq \alpha' \in \overline{\R}$, and let $Q$ be a quad which is an $\alpha$-quad and an $\alpha'$-quad at the same time. Then $Q$ is a kite.
\end{proposition}

This generalizes a result of \cite{Josefsson2} which claims that a quad which is both a $0$-quad and a $1$-quad is a kite.

\begin{proof}
Denote by $\ell_1,\ell_2,\ell_3,\ell_4$ the side lengths of $Q$ in cyclic order. Assume first that neither $\alpha$ nor $\alpha'$ take the values $-\infty,\infty$ or $0$. Then, up to replacing $\ell_i$ by $\ell_i^{\alpha'}$, one may assume that $\alpha'=1$. Write $s=\ell_1+\ell_3$ and $t=\ell_1^\alpha + \ell_3^\alpha$. Since $Q$ is a $1$-quad, we have $\ell_4=s-\ell_2$. Since $Q$ is also an $\alpha$-quad, we have
\begin{equation}
\label{eq:intermediatevalue}
\ell_2^\alpha + (s-\ell_2)^\alpha = t,
\end{equation}
and $0 \leq \ell_2 \leq s$. The function $x \mapsto x^\alpha + (s-x)^\alpha$ is strictly monotonic on $[0,s/2]$ symmetric with respect to $x=s/2$, so that the only solutions of equation~\eqref{eq:intermediatevalue} are $\ell_2=\ell_1$ and $\ell_2=\ell_3$. Hence $\{\ell_1,\ell_3\}=\{\ell_2,\ell_4\}$, and $Q$ is a kite.

Assume now that $\alpha=0$ and $\alpha'$ is finite. One may again assume that $\alpha'=1$, in which case the sums and products of each pair $\{l_1,l_3\}$ and $\{l_2,l_4\}$ are equal, and $Q$ is again a kite.
If $\alpha \in \{-\infty,\infty\}$ and $\alpha'$ is finite non-zero, one may again assume that $\alpha'=1$ and the conclusion follows easily. The case when $\alpha \in \{-\infty,\infty\}$ and $\alpha'=0$ is similar. Finally the case when $\{\alpha,\alpha'\}=\{-\infty,\infty\}$ is also easy.
\end{proof}

We finally provide another characterization, involving circumradii of
the four triangles delimited by the sides and diagonals of the
quad. This last characterization actually holds in a more generic setting, which we now introduce. In what follows, we say that a symmetric function of two variables $f:(\R_+)^2\rightarrow\R$ is \emph{homogeneous} if there exists $u\in\R$ such that for any $\lambda,x,y>0$, we have $f(\lambda x,\lambda y)=\lambda^u f(x,y)$.

\begin{definition}
\label{def:fquad}
Let $f:(\R_+)^2\rightarrow\R$ be a non-constant homogeneous symmetric function of two variables. A quadrilateral $ABCD$ is called an \emph{$f$-quadrilateral} if:
\begin{equation}
\label{eq:fquadlengths}
f(AB,CD)=f(BC,AD).
\end{equation}
\end{definition}

Multiplying $f$ by a non-zero scalar, or more generally post-composing it with a bijection, produces the same class of quads. As an example, for $\alpha \in \R^*$, $\alpha$-quads correspond to the function $f_\alpha: (x,y) \mapsto x^\alpha+y^\alpha$. By definition, $0$-quads, $+\infty$-quads an $-\infty$-quads are also $f$-quadrilaterals for some homogeneous function.

\begin{proposition}
\label{prop:circumrad}
Let $ABCD$ be a quad. Let $P$ denote the intersection point of its
diagonals and suppose that $P$ is distinct from $A,B,C,D$. We denote the respective circumradii of the triangles $ABP,BCP,CDP$ and $DAP$ by $R_1,R_2,R_3$ and $R_4$. Let $f:(\R_+)^2\rightarrow\R$ be a symmetric homogeneous function. The following are equivalent:
\begin{enumerate}
 \item $ABCD$ is an $f$-quadrilateral ;
 \item $f(R_1,R_3)=f(R_2,R_4)$.
\end{enumerate}
\end{proposition}

This result was already known for orthodiagonal quads and tangentials quads \cite{Josefsson} and our proof below is a straightforward extension of the proof of \cite[Theorem 9]{Josefsson}.

\begin{proof}
Let us denote the centers of the circumcircles of $ABP,BCP,CDP,DAP$ by respectively $O_1,O_2,O_3,O_4$. Since $AO_1B$ is isoceles in $O_1$, we have
\[
AB=2R_1\sin\frac{\widehat{AO_1B}}{2}.
\]
Since $P$ lies on the circle centered at $O_1$ and going through $A$ and $B$, we have that
\[
\frac{\widehat{AO_1B}}{2} \in \left\{\widehat{APB},\pi-\widehat{APB} \right\},
\]
hence
\begin{equation}
AB=2R_1\sin \widehat{APB}.
\end{equation}
Similarly we have
\begin{align}
BC&=2R_2\sin \widehat{BPC} \\
CD&=2R_3\sin \widehat{CPD} \\
DA&=2R_4\sin \widehat{DPA}
\end{align}
Observing that $\widehat{APB}=\widehat{CPD}=\pi-\widehat{BPC}=\pi-\widehat{DPA}$, we deduce that the quadruples $(AB,BC,CD,DA)$ and $(R_1,R_2,R_3,R_4)$ are proportional. Since $f$ is homogeneous, $f(AB,CD)=f(BC,DA)$ if and only if $f(R_1,R_3)=f(R_2,R_4)$. This concludes the proof.
\end{proof}

\begin{remark}
Let $h_1$, $h_2$, $h_3$ and $h_4$ be the altitudes through $P$ in the triangles $ABP$, $BCP$, $CDP$ and $DAP$. For $\alpha\in\{1,2\}$, $ABCD$ is an $\alpha$-quad if and only if
\[
h_1^{-\alpha} + h_3^{-\alpha} = h_2^{-\alpha} + h_4^{-\alpha},
\]
see \cite[Section 5]{Josefsson}. However this does not hold for general values of $\alpha$.
\end{remark}

We end this section with an open question.
\begin{question}
Apart from $f_\alpha$, are there homogeneous functions $f$ such that $f$-quadrilaterals satisfy one of the flip properties discussed in this paper?
\end{question}

\section*{Acknowledgements}

We are grateful to Dmitry Chelkak, Marianna Russkikh and Andrea
Sportiello for numerous fruitful discussions. We also thank Thomas
Blomme and Othmane Safsafi for discussions leading to the proof of
Propositions~\ref{prop:intersection} and
\ref{prop:circumrad}. We thank the referees for comments
leading to an improved exposition. PM and SR are partially supported by the Agence Nationale de la Recherche, Grant Number ANR-18-CE40-0033 (ANR DIMERS). SR also acknowledges the support of the Fondation Sciences Math\'ematiques de Paris.
PT acknowledges partial support from Agence Nationale de la Recherche, Grant Number ANR-14-CE25-0014 (ANR GRAAL).

\appendix

\section{Jacobi elliptic functions}
\label{sec:elliptic}

In this appendix we give a very brief presentation of elliptic
functions and formulas used in the paper, with variable names that are
not classical but fit the framework of our paper. For a more complete
introduction, we refer to \cite{Lawden, AbramowitzStegun}.

Consider a number $k \in [0,1)\cup i\R$, which we call \emph{elliptic
  modulus}. We first define the
\emph{incomplete integral of the first kind}:
\begin{equation}
  F(\theta,k) = \int_{0}^{\theta} \frac{dt}{\sqrt{1-k^2\sin^2t}}.
\end{equation}
The function $F(\cdot,k)$ is a bijection from $\R$ to $\R$, and we
denote its inverse by $\am(\cdot,k)$. Then the first two Jacobi
elliptic functions are defined for any $\tau\in \R$ by
\begin{equation}
  \begin{split}
    \cn(\tau,k) &= \cos(\am(\tau,k)), \\
    \sn(\tau,k) &= \sin(\am(\tau,k)).
  \end{split}
\end{equation}
These functions can be thought of as generalizations of usual cosine
and sine functions. For $k=0$, the function $F(\cdot,k)$ is clearly
the identity, hence $\cn$ and $\sn$ reduce to $\cos$ and
$\sin$. In general, they are periodic functions, of period $4K(k)$
where $K$ denotes the \emph{complete integral of the first kind}, also
named \emph{quarter-period}:
\begin{equation}
  K(k) = F\left(\frac{\pi}{2},k\right).
\end{equation}

Apart from $\cn$ and $\sn$, we also define
\begin{equation}
  \dn(\tau,k) = \sqrt{1-k^2\sn(\tau,k)}.
\end{equation}
Then, for any two distinct letters $p,q \in \{c,s,d,n\}$, we define a
function $pq$ as $\frac{pn}{qn}$, with the convention that $nn=1$. For instance,
\begin{equation}
  \begin{split}
    \nc(\tau,k) &= \frac{1}{\cn(\tau,k)}, \\
    \sc(\tau,k) &= \frac{\sn(\tau,k)}{\cn(\tau,k)},
  \end{split}
\end{equation}
etc.

The following formulas, which can
be found in \cite{AbramowitzStegun}, chapter 16.8, describe how these elliptic functions are affected by the change of variable $\tau\mapsto K(k)-\tau$. As is
customary, we omit the dependence in $k$ to simplify notations. Letting $k'=\sqrt{1-k^2}$, one has:
\begin{equation}
  \label{eq:ell_change_var}
  \begin{split}
    \cn(K-\tau) &= k' \sd(\tau), \\
    \sn(K-\tau) &= \cd(\tau), \\
    \dn(K-\tau) &= k' \nd(\tau).
  \end{split}
\end{equation}
The effect of the same change of variable on any other function $pq$ can
be deduced from these three.

We also make use of an addition identity, that can be found as
formula~32(i) in Chapter~2 of \cite{Lawden}. For any $\tau,\tau' \in \R,$
\begin{equation}
  \label{eq:ell_add}
  \cn({\tau}+{\tau'}) +
  \sn({\tau})\sn({\tau'})
  \dn({\tau}+{\tau'}) -
  \cn({\tau})\cn({\tau'}) = 0.
\end{equation}

\label{Bibliography}
\bibliographystyle{plain}
\bibliography{bibliographie}

\begin{thebibliography}{10}

\bibitem{AbramowitzStegun}
Milton Abramowitz and Irene~A. Stegun.
\newblock {\em Handbook of Mathematical Functions with Formulas, Graphs, and
  Mathematical Tables}.
\newblock Dover, New York, ninth {D}over printing, tenth {GPO} printing
  edition, 1964.

\bibitem{Atzema}
Eisso~J. Atzema.
\newblock A theorem by {G}iusto {B}ellavitis on a class of quadrilaterals.
\newblock {\em Forum Geom.}, 6:181--185, 2006.

\bibitem{Baxter:exactly}
Rodney~J. Baxter.
\newblock {\em Exactly solved models in statistical mechanics}.
\newblock Academic Press, London, 1982.

\bibitem{BergerBiskup}
Noam Berger and Marek Biskup.
\newblock Quenched invariance principle for simple random walk on percolation
  clusters.
\newblock {\em Probab. Theory Related Fields}, 137(1-2):83--120, 2007.

\bibitem{BoutillierDeTiliereRaschel}
C\'{e}dric Boutillier, B\'{e}atrice de~Tili\`ere, and Kilian Raschel.
\newblock The {$Z$}-invariant {I}sing model via dimers.
\newblock {\em Probab. Theory Related Fields}, 174(1-2):235--305, 2019.

\bibitem{Chelkak}
Dmitry Chelkak.
\newblock Planar {I}sing model at criticality: state-of-the-art and
  perspectives.
\newblock In {\em Proceedings of the {I}nternational {C}ongress of
  {M}athematicians---{R}io de {J}aneiro 2018. {V}ol. {IV}. {I}nvited lectures},
  pages 2801--2828. World Sci. Publ., Hackensack, NJ, 2018.

\bibitem{Chelkak2}
Dmitry Chelkak.
\newblock Ising model and s-embeddings of planar graphs.
\newblock {\em arXiv preprint}, 2020.
\newblock arXiv:2006.14559.

\bibitem{ChelkakSmirnov1}
Dmitry Chelkak and Stanislav Smirnov.
\newblock Discrete complex analysis on isoradial graphs.
\newblock {\em Adv. Math.}, 228(3):1590--1630, 2011.

\bibitem{ChelkakSmirnov2}
Dmitry Chelkak and Stanislav Smirnov.
\newblock Universality in the 2{D} {I}sing model and conformal invariance of
  fermionic observables.
\newblock {\em Invent. Math.}, 189(3):515--580, 2012.

\bibitem{Dubedat}
Julien {Dub{\'e}dat}.
\newblock {Exact bosonization of the Ising model}.
\newblock {\em arXiv preprint}, 2011.
\newblock arXiv:1112.4399.

\bibitem{DCLM}
Hugo Duminil-Copin, Jhih-Huang Li, and Ioan Manolescu.
\newblock Universality for the random-cluster model on isoradial graphs.
\newblock {\em Electron. J. Probab.}, 23:70 pp., 2018.

\bibitem{GrimmettManolescu}
Geoffrey~R. Grimmett and Ioan Manolescu.
\newblock Bond percolation on isoradial graphs: criticality and universality.
\newblock {\em Probab. Theory Related Fields}, 159(1-2):273--327, 2014.

\bibitem{Josefsson2}
Martin Josefsson.
\newblock When is a tangential quadrilateral a kite?
\newblock {\em Forum Geom.}, 11, 2011.

\bibitem{Josefsson}
Martin Josefsson.
\newblock Characterizations of orthodiagonal quadrilaterals.
\newblock {\em Forum Geom.}, 12:13--25, 2012.

\bibitem{Kakutani}
Shizuo Kakutani.
\newblock Markoff process and the {D}irichlet problem.
\newblock {\em Proc. Japan Acad.}, 21:227--233 (1949), 1945.

\bibitem{Kennelly}
Arthur~E. Kennelly.
\newblock Equivalence of triangles and star in conducting networks.
\newblock {\em Electr. World and Eng.}, 34:413--414, 1899.

\bibitem{Kenyon:intro}
Richard Kenyon.
\newblock An introduction to the dimer model.
\newblock In {\em School and {C}onference on {P}robability {T}heory}, ICTP
  Lect. Notes, XVII, pages 267--304. Abdus Salam Int. Cent. Theoret. Phys.,
  Trieste, 2004.

\bibitem{KLRR}
Richard {Kenyon}, Wai~Yeung {Lam}, Sanjay {Ramassamy}, and Marianna {Russkikh}.
\newblock Dimers and circle patterns.
\newblock {\em Ann. Sci. \'{E}c. Norm. Sup\'{e}r.}, 2021.
\newblock To appear.

\bibitem{KenyonPemantle_minors}
Richard Kenyon and Robin Pemantle.
\newblock Principal minors and rhombus tilings.
\newblock {\em J. Phys. A}, 47(47):474010, 17, 2014.

\bibitem{KenyonPemantle}
Richard Kenyon and Robin Pemantle.
\newblock Double-dimers, the {I}sing model and the hexahedron recurrence.
\newblock {\em J. Combin. Theory Ser. A}, 137:27--63, 2016.

\bibitem{KS}
Boris~G. Konopelchenko and Wolfgang~K. Schief.
\newblock Reciprocal figures, graphical statics, and inversive geometry of the
  {S}chwarzian {BKP} hierarchy.
\newblock {\em Stud. Appl. Math.}, 109(2):89--124, 2002.

\bibitem{KramersWannier}
Hans~A. Kramers and Gregory~H. Wannier.
\newblock Statistics of the two-dimensional ferromagnet. {I}.
\newblock {\em Phys. Rev. (2)}, 60:252--262, 1941.

\bibitem{Lawden}
Derek~F. Lawden.
\newblock {\em Elliptic Functions and Applications}.
\newblock Applied Mathematical Sciences. Springer New York, 1989.

\bibitem{Lis}
Marcin Lis.
\newblock Circle patterns and critical {I}sing models.
\newblock {\em Comm. Math. Phys.}, 370(2):507--530, 2019.

\bibitem{LSFW}
Dennis Lucarelli, Anshu Saksena, Ryan Farrell, and I-Jeng Wang.
\newblock Distributed inference for network localization using radio
  interferometric ranging.
\newblock {\em Wirel. Sens. Netw.}, pages 52--73, 2008.

\bibitem{Onsager}
Lars Onsager.
\newblock Crystal statistics. {I}. {A} two-dimensional model with an
  order-disorder transition.
\newblock {\em Phys. Rev. (2)}, 65:117--149, 1944.

\bibitem{Tutte:emb}
William~T. Tutte.
\newblock How to draw a graph.
\newblock {\em Proc. London Math. Soc. (3)}, 13:743--767, 1963.

\bibitem{Wannier}
Gregory~H Wannier.
\newblock The statistical problem in cooperative phenomena.
\newblock {\em Rev. of Modern Phys.}, 17(1):50, 1945.

\bibitem{Whitney}
Hassler Whitney.
\newblock 2-{I}somorphic {G}raphs.
\newblock {\em Amer. J. Math.}, 55(1-4):245--254, 1933.

\bibitem{XSR}
Xiaochun Xu, Sartaj Sahni, and Nageswara S.~V. Rao.
\newblock On basic properties of localization using distance-difference
  measurements.
\newblock In {\em 2008 11th International Conference on Information Fusion},
  pages 1--8, 2008.

\end{thebibliography}
\end{document}